\documentclass[10pt]{amsart}
\usepackage{amssymb, amsthm, amsmath}
\usepackage[all]{xy}
\usepackage{graphicx}
\usepackage{psfrag}

\newtheorem{prop}{Proposition}
\newtheorem{thm}{Theorem}
\newtheorem{cor}{Corollary}
\newtheorem{lemma}{Lemma}
\theoremstyle{definition}
\newtheorem{defn}{Definition}
\newtheorem{example}{Example}

\newcommand\A{{\mathbb A}}

\newcommand\C{{\mathbb C}}

\newcommand\N{{\mathbb N}}
\newcommand{\cN}{{\mathcal N}}

\newcommand{\ti}{\vartheta}
\newcommand{\Ti}{\Theta}

\newcommand\cC{{\mathcal C}}
\newcommand\cB{{\mathcal B}}

\newcommand\Z{{\mathbb Z}}
\newcommand\cP{{\mathcal P}}
\newcommand\cQ{{\mathcal Q}}

\newcommand\cR{{\mathcal R}}

\newcommand\CS{{\mathfrak C}}

\newcommand\al{\alpha}
\newcommand{\be}{\beta}
\newcommand\la{\lambda}
\newcommand\s{{\sigma}}
\newcommand\om{{\varpi}}

\newcommand\Sh{{\mathcal S}}

\newcommand\ssm{\smallsetminus}
\newcommand\gequ{\geq}
\newcommand\lequ{\leq}
\newcommand\noin{\noindent}

\newcommand\eqto{\stackrel{\lower1.5pt\hbox{$\scriptstyle\sim\,$}}\to}
\newcommand\ov{\overline}

\newcommand\rsa{\rightsquigarrow}
\newcommand\wh{\widehat}
\newcommand\wt{\widetilde}
\newcommand\dis{\displaystyle}

 \DeclareMathOperator{\Pf}{Pfaffian}

\DeclareMathOperator{\LG}{LG} \DeclareMathOperator{\IG}{IG}
 
\DeclareMathOperator{\G}{G} 
\DeclareMathOperator{\HH}{\mathrm{H}}

\newcommand{\ignore}[1]{}
\newcommand{\pic}[2]{\includegraphics[scale=#1]{#2}}

\psfrag{k}{$k$}
\psfrag{(r,c)}{$[r,c]$}
\psfrag{(r',c')}{$[r',c']$}

\begin{document}

\title[Giambelli, Pieri, and tableau formulas via raising operators]
{Giambelli, Pieri, and tableau formulas via raising~operators}

\date{January 20, 2011}

\author{Harry~Tamvakis} \address{University of Maryland, Department of
Mathematics, 1301 Mathematics Building, College Park, MD 20742, USA}
\email{harryt@math.umd.edu}

\subjclass[2000]{Primary 05E15; Secondary 14M15, 14N15, 05E05}

\thanks{The author was supported in part by 
  NSF Grants DMS-0639033 and DMS-0901341.}

\begin{abstract}
We give a direct proof of the equivalence between the Giambelli and
Pieri type formulas for Hall-Littlewood functions using Young's
raising operators, parallel to joint work with Buch and Kresch for the
Schubert classes on isotropic Grassmannians. We prove several closely
related mirror identities enjoyed by the Giambelli polynomials, which
lead to new recursions for Schubert classes.  The raising operator
approach is applied to obtain tableau formulas for the Hall-Littlewood
functions, the theta polynomials of \cite{BKT2}, and related Stanley
symmetric functions.  Finally, we introduce the notion of a skew
element $w$ of the hyperoctahedral group and identify the set of
reduced words for $w$ with the set of standard $k$-tableaux on a skew
shape $\la/\mu$.
\end{abstract}

\maketitle

\setcounter{section}{-1}

\section{Introduction}
The classical Schubert calculus is concerned with the algebraic
structure of the cohomology ring of the Grassmannian $\G(m,N)$ of
$m$-dimensional subspaces of complex affine $N$-space. The cohomology
has a free $\Z$-basis of Schubert classes $\s_{\la}$, induced by the
natural decomposition of $\G(m,N)$ into a disjoint union of Schubert
cells. On the other hand, the ring is generated by the Chern classes
of the universal quotient bundle $\cQ$ over $\G(m,N)$, also known as
the {\em special} Schubert classes. The theorems of Pieri \cite{Pi}
and Giambelli \cite{G} are fundamental building blocks of the subject:
the former is a rule for a product of a general Schubert class
$\s_{\la}$ with a special one, while the latter is a formula equating
$\s_{\la}$ with a polynomial in the special classes.  When one
expresses the Chern classes involved in terms of the Chern roots of
$\cQ$, the Schubert classes are replaced by Schur $S$-polynomials,
thus exhibiting a link between the Schubert calculus and the ring of
symmetric functions.

In a series of papers with Buch and Kresch \cite{BKT1, BKT2}, we
obtained corresponding results for the Grassmannians parametrizing
(non-maximal) isotropic subspaces of a vector space equipped with a
nondegenerate symmetric or skew-symmetric bilinear form.  Our
solution of the Giambelli problem for isotropic Grassmannians uses the
{\em raising operators} of Young \cite{Y} in an essential way; the
resulting formula interpolates naturally between a Jacobi-Trudi
determinant and a Schur Pfaffian. A rather different context in which
a Giambelli type formula appears that has this interpolation property
is the theory of Hall-Littlewood symmetric functions \cite{Li, M};
these objects also satisfy a Pieri rule \cite{Mo}.

The raising operator approach allows one to see {\em directly} that
the Pieri and Giambelli results are formally {\em equivalent} to each
other in all the above instances. This amounts to showing that {\em
the Giambelli polynomials satisfy the Pieri rule}, for the converse
implication then follows easily. As a consequence, working either in
the context of Schubert calculus or the theory of symmetric functions,
it suffices to prove either of the two theorems to establish them
both. In geometry, the Pieri rule may be proved concisely by studying
a triple intersection of the appropriate Schubert cells, following
Hodge \cite{H}; this is the method used in \cite{BKT1}.

When working algebraically, it is convenient to use the Giambelli
formula as the starting point, and this is the point of view of the
present article. One advantage of the raising operator definition is
that it makes clear that the basic Giambelli polynomials, which are
indexed by partitions $\la$, make sense when the index is any finite
sequence of integers, positive or negative. When the index is a
partition, it is the axis of reflection for a pair of {\em mirror
identities}, both closely connected to the Pieri rule. The 
simplest example is in the case of Schur $S$-functions:
\[
\sum_{\al\geq 0} s_{\la+\al} = \sum_{\mu\supset\la}s_\mu
\quad \text{ and } \quad
\sum_{\al\geq 0} s_{\la-\al} = \sum_{\mu\subset\la}s_\mu
\]
where the sums are over all compositions $\al$ and partitions $\mu$
obtained from $\la$ by adding (respectively subtracting) a horizontal
strip.  In \S \ref{recsec} we use raising operators and the mirror
identities to obtain top row recursion formulas for Schubert classes
on isotropic Grassmannians.

Our central thesis up to this point is that the aforementioned results
may be proved without recourse to their realizations in terms of
symmetric functions. The main application of the mirror identities,
however, lies in the latter theory, where they can be used to obtain
{\em reduction formulas} for the number of variables $x$ which appear
in the argument of a symmetric polynomial. These in turn lead directly
to {\em tableau formulas} for the polynomials in question. The tableau
formulas suggest the introduction of new symmetric functions indexed
by {\em skew Young diagrams}, and one can then study their
properties. In this way, many important aspects of the classical
theory of symmetric functions, as presented in the first three
chapters of \cite{M}, may be established by entirely different
methods. Moreover, the intention is to apply these techniques in 
a hitherto unexplored situation.

Our main result is a tableau formula for the theta polynomials of
\cite{BKT2}, which are the Billey-Haiman type C Schubert polynomials 
\cite{BH} for Grassmannian elements of the hyperoctahedral group
$B_\infty$ (the union of the Weyl groups for the root systems of type
$\text{B}_n$ or $\text{C}_n$). The theorem states that for any
$k$-strict partition $\la$, we have
\begin{equation}
\label{tabeq}
\Ti_\la(x\,;y) = \sum_U 2^{n(U)}(xy)^U 
\end{equation}
where the sum is over all $k$-bitableaux $U$ of shape $\la$ (see \S
\ref{tableaux} for the precise definitions).  It bears emphasis that
our proof of (\ref{tabeq}) does not use inner products, tableau
correspondences, or algorithms such as jeu de taquin and the like. We
show that formula (\ref{tabeq}) (and its counterpart in the theory of 
Schur and Hall-Littlewood functions) follows essentially from the
definition of $\Ti_\la(x\,;y)$ using raising operators, once the Pieri
rule is established. The result is surprising because the
Pieri rule for the product $\ti_p\cdot \Ti_\la$ involves partitions
$\mu$ which do not contain the diagram of $\la$.

We are led to introduce symmetric
functions $F^{(k)}_{\la/\mu}(x)$, defined for any pair $\mu\subset
\la$ of $k$-strict partitions by the equation
\[
F^{(k)}_{\la/\mu}(x) = \sum_T 2^{n(T)} x^T
\]
where the sum is over all $k$-tableaux $T$ of skew shape $\la/\mu$.
When $k=0$, the $F^{(k)}_{\la/\mu}$ coincide with the usual skew Schur
$Q$-functions $Q_{\la/\mu}$. For general $k$, we prove that
$F^{(k)}_{\la/\mu}$ either vanishes or is equal to a certain type C
Stanley symmetric function $F_w$, and therefore is always a
nonnegative integer linear combination of Schur $Q$-functions. The
function $F_w$ is indexed by an element $w$ of the hyperoctahedral
group which we call {\em skew}. We argue that the skew elements of
$B_\infty$ are directly analogous to the 321-avoiding permutations
\cite{BJS} in the symmetric group. For example, we show that the
reduced words for $w$ are in 1-1 correspondence with the standard
$k$-tableaux of shape $\la/\mu$.

We begin this article in \S \ref{typeA} by giving a short proof using
raising operators of the equivalence between the classical Giambelli
formula and the Pieri rule. This is followed by a discussion of the
analogous result for the algebra of Hall-Littlewood functions in \S
\ref{HL}, and then the more sophisticated arguments of \cite{BKT2} in
\S \ref{typeC:general}.  The raising operators $R^\la$ which appear in
the Giambelli formulas of isotropic Schubert calculus depend on the
partition $\lambda$; this leads to a more challenging and dynamic
theory. In \S \ref{mirrorsec} we obtain the mirror identities for
Hall-Littlewood functions and isotropic Grassmannian Schubert classes,
and give some first applications. Section \ref{red&tab} contains our
treatment of the various tableau formulas. In \S \ref{tabHL} we stop
short of proving the (known) tableau formula for skew Hall-Littlewood
functions, since the more difficult case of theta polynomials and the
skew functions $F^{(k)}_{\la/\mu}(x)$ is handled in detail in \S
\ref{theta1} -- \S \ref{theta4}. Finally, in \S \ref{sss} we define
the skew elements of the hyperoctahedral group and relate the results
of \S \ref{red&tab} to the Billey-Haiman Schubert polynomials and type
C Stanley symmetric functions.

For simplicity we will restrict to the type A and C theories
throughout the article, and not discuss the analogous results for
orthogonal Grassmannians and Weyl groups here. We hope that \S
\ref{typeA} and the examples of \S \ref{pierirulegen} provide a useful
introduction to \cite{BKT2}. To emphasize that the formulas in \S
\ref{typeA} - \S \ref{mirrorsec} do not depend on their specific
realizations in the Schubert calculus or the theory of symmetric
functions, we have used the nonstandard notation $U_\la$, $V_\la$,
$W_\la$ for the basic polynomials, and in \S \ref{oddsendsA}, \S
\ref{oddsendsHL}, and \S \ref{oddsendsCgen} briefly discuss how these
relate to other known objects.

Many thanks are due to my collaborators Anders Buch and Andrew
Kresch for their inspired work on related projects.

\section{The type A theory}
\label{typeA}

\subsection{Raising operators} An {\em integer sequence} or 
{\em integer vector} is a sequence of integers $\{\al_i\}_{i\geq 1}$,
only finitely many of which are non-zero. The largest integer
$\ell\geq 0$ such that $\al_\ell\neq 0$ is called the {\em length} of
$\al$, denoted $\ell(\al)$; we will identify an integer sequence of
length $\ell$ with the vector consisting of its first $\ell$ terms.
We let $|\al| = \sum \al_i$ and $\#\al$ equal the number of non-zero
parts $\al_i$ of $\al$. We write $\al\geq \be$ if $\al_i \geq \be_i$
for each $i$.  An integer vector $\al$ is a {\em composition} if
$\al_i\geq 0$ for all $i$ and a {\em partition} if $\al_i \geq
\al_{i+1}\geq 0$ for all $i$. As is customary, we represent a
partition $\la$ by its Young diagram of boxes, which has $\la_i$ boxes
in row $i$. We write $\mu\subset\la$ instead of $\mu\leq\la$ for the
containment relation between two Young diagrams; in this case the
set-theoretic difference $\la\ssm\mu$ is called a skew diagram and
denoted $\la/\mu$.

For two integer sequences $\al$, $\be$ such that $|\al|= |\be|$, we
say that $\al$ {\em dominates} $\be$ and write $\al\succeq \be$ if
$\al_1+\cdots + \al_i\geq \be_1+\cdots +\be_i$ for each $i$.  Given
any integer sequence $\alpha=(\alpha_1,\alpha_2,\ldots)$ and $i<j$, we
define
\[
R_{ij}(\alpha) = (\alpha_1,\ldots,\alpha_i+1,\ldots,\alpha_j-1,
\ldots);
\] 
a raising operator $R$ is any monomial in these
$R_{ij}$'s. Note that we have $R\,\al\succeq \al$ for any integer sequence
$\al$; conversely, if $\al\succeq \be$, there exists a
raising operator $R$ such that $\al = R\,\be$. See 
\cite[I.1]{M} for more information.

\subsection{The Giambelli formula} 

Consider the polynomial ring $A=\Z[u_1,u_2,\ldots]$ where the $u_i$
are countably infinite commuting independent variables. We regard $A$
as a graded ring with each $u_i$ having graded degree $i$, and
adopt the convention here and throughout the paper that $u_0=1$ while
$u_r=0$ for $r<0$. For each integer vector $\al$, set $u_{\al} =
\prod_iu_{\al_i}$; then $A$ has a free $\Z$-basis consisting of the
monomials $u_{\la}$ for all partitions $\la$. 

Given a raising operator $R$, define $R\,u_{\al} = u_{R\al}$. Here we
view the raising operator $R$ as acting on the index $\al$, and not on
the monomial $u_\al$ itself. Thus, if the components of $\al$ are a
permutation of the components of $\be$, it may happen that $R\,u_{\al}
\neq R\, u_{\be}$ even though $u_\al = u_\be$ as elements of
$A$. Notice that if $\al_\ell<0$ for $\ell=\ell(\al)$, then 
$R\,u_\al=0$ in $A$ for any raising operator $R$.

For any integer vector $\al$, define $U_{\al}$ by the {\em Giambelli
formula}
\begin{equation}
\label{giambelliA}
U_{\al} := \prod_{i<j} (1-R_{ij})\, u_{\al}.
\end{equation}
Although the product in (\ref{giambelliA}) is infinite, if we expand
it into a series of terms we see that only finitely many of the
summands are non-zero; hence, $U_{\al}$ is well defined.  For any
partition $\la$, we clearly have
\[
U_{\la} = u_{\la} + \sum_{\mu\succ\la}a_{\la\mu} u_\mu
\]
where $a_{\la\mu}\in\Z$ and the sum is over partitions $\mu$ which
strictly dominate $\la$.  We deduce that the $U_{\la}$ for $\la$ a
partition form another $\Z$-basis of $A$.

We have $U_{(r,s)} = (1-R_{12})\, u_{(r,s)} = 
u_ru_s - u_{r+1}u_{s-1}$, hence $U_{(r,s)} = - U_{(s-1,r+1)}$ for
any $r,s\in\Z$. This property has a straightforward generalization.

\begin{lemma} \label{commuteA}
Let $\al$ and $\be$ be integer vectors.  
Then for any $r,s\in\Z$ we have
\[ U_{(\al,r,s,\be)} = - U_{(\al,s-1,r+1,\be)} \,. \]
\end{lemma}
\noin
We postpone the proof until \S \ref{HL}, where we derive a more general
result in the context of the Hall-Littlewood theory (Lemma
\ref{commuteHL}).

\subsection{The Pieri rule} 
\label{pieriruleA}

For any $d\geq 1$ define the raising operator $R^d$ by
\[
R^d = \prod_{1\leq i<j \leq d} (1-R_{ij}).
\]
For $p>0$ and any partition $\la$ of length $\ell$, we compute 
\[
u_p\cdot U_\la = u_p\cdot R^{\ell}\, u_{\la} =
R^{\ell+1}\cdot
\prod_{i=1}^\ell(1-R_{i,\ell+1})^{-1} \, u_{(\la,p)}
\]
\[
= R^{\ell+1}\cdot\prod_{i=1}^\ell(1+R_{i,\ell+1} + 
R_{i,\ell+1}^2 + \cdots)\,u_{(\la,p)} 
\]
and therefore
\begin{equation}
\label{initeq}
u_p\cdot U_\la = \sum_{\nu\in \cN} U_\nu,
\end{equation}
where $\cN=\cN(\la,p)$ is the set of all compositions $\nu\geq \la$
such that $|\nu| = |\la|+p$ and $\nu_j =0$ for
$j > \ell+1$. 

Call a composition $\nu\in \cN$ {\em bad} if there exists
a $j>1$ such that $\nu_j > \la_{j-1}$, and let $X$ be the set of all
bad compositions. Define an involution $\iota:X\to X$ as follows: 
for $\nu\in X$, choose $j$ maximal such that $\nu_j > \la_{j-1}$, and 
set
\[
\iota(\nu) = (\nu_1,\ldots,\nu_{j-2}, \nu_j-1,\nu_{j-1}+1,
\nu_{j+1},\ldots,\nu_{\ell+1}).
\]
Lemma \ref{commuteA} implies that $U_\nu + U_{\iota(\nu)} = 0$
for every $\nu\in X$, therefore all bad indices may be omitted from
the sum in (\ref{initeq}). We are left with the 
{\em Pieri rule}
\begin{equation}
\label{pieriA}
u_p\cdot U_\la = \sum_\nu U_\nu
\end{equation}
summed over all partitions $\nu$ with $|\nu| = |\la|+p$ such that 
$\nu_1 \geq \la_1 \geq \nu_2 \geq \la_2 \geq \cdots$. In the language
of Young diagrams, this condition means that $\nu\supset\la$ and the 
skew diagram $\nu/\la$ is a {\em horizontal $p$-strip}.

Conversely, suppose that we are given a family $\{T_\la\}$ of elements
of $A$, one for each partition $\la$, such that $T_p=u_p$ for every
integer $p\geq 0$ and the $T_\la$ satisfy the Pieri rule $T_p\cdot
T_\la = \sum_{\nu} T_\nu$, with the sum over $\nu$ as in
(\ref{pieriA}). We claim then that
\[
T_\la = U_\la = \prod_{i<j}(1-R_{ij})\, u_\la
\]
for every partition $\la$. To see this, note that the Pieri rule
implies that
\[
U_\la + \sum_{\mu\succ\la} a_{\la\mu} \, U_\mu = 
u_{\la_1}\cdots u_{\la_\ell} = 
T_\la + \sum_{\mu \succ \la} a_{\la\mu}\, T_\mu
\]
for some constants $a_{\la\mu}\in\Z$. The claim now follows by induction
on $\la$. We thus see that {\em the Giambelli formula and Pieri rule
are formally equivalent to each other}.

\subsection{The Grassmannian and symmetric functions}
\label{oddsendsA}
Equation (\ref{giambelliA}) may be written in the more
familiar form
\begin{equation}
\label{giambelliA2}
U_{\al} = \det(u_{\al_i+j-i})_{i,j}.
\end{equation}
We will prove that (\ref{giambelliA}) and (\ref{giambelliA2}) are
equivalent, following \cite[I.3]{M}. Suppose the length of $\al$ is
$\ell$ and let $\rho = (0,1,\ldots, \ell-1)$.  We work in the ring
$\Z[x_1,x_1^{-1},\ldots, x_{\ell},x_{\ell}^{-1}]$ and let $R\, x^\al =
x^{R\al}$ for any raising operator $R$. The key is to use the {\em
Vandermonde identity}
\[
\det (x_i^{j-1})_{1\leq i,j\leq \ell} = 
\prod_{1\leq i < j \leq \ell}(x_j - x_i)
\]
which implies that 
\[
\prod_{1\leq i<j\leq \ell} (1-R_{ij})\, x^{\al} = 
\prod_{1\leq i<j\leq \ell} (1-x_ix_j^{-1})\, x^{\al} = 
x^{\al-\rho}\det (x_i^{j-1})_{1\leq i,j\leq \ell} 
\]
\[
= \sum_{\s\in S_\ell}(-1)^{\s}x_1^{\al_1+\s(1)-1}\cdots
x_\ell^{\al_\ell+\s(\ell)-\ell}
= \det (x_i^{\al_i+j-i})_{1\leq i,j\leq \ell}.
\]
One now applies the $\Z$-module homomorphism $\Z[x_1,x_1^{-1},\ldots, 
x_{\ell},x_{\ell}^{-1}]\to A$ sending $x^\al$ to $u_\al$ for each $\al$
to both ends of the above equation to obtain the result.

The two standard realizations of the Pieri and Giambelli formulas in
the literature are the cohomology ring of the Grassmannian $\G=\G(m,N)$
and the ring $\Lambda$ of symmetric functions. The ring $\HH^*(\G,\Z)$
has a free $\Z$-basis of Schubert classes $\s_\la$, one for each
partition $\la$ whose diagram is contained in an $m\times (N-m)$
rectangle $\cR(m,N-m)$.  There is a ring epimorphism $\phi: A \to
\HH^*(\G,\Z)$ sending the generators $u_p$ to the special Schubert
classes $\s_p$ for $1\leq p \leq N-m$ and to zero for $p > N-m$. The
map $\phi$ satisfies $\phi(U_\la) = \sigma_\la$ if $\la \subset
\cR(m,N-m)$, and $\phi(U_\la) = 0$ otherwise. For further details,
the reader may consult e.g.\ \cite[Chapter 9.4]{Fu}. 

Let $x=(x_1,x_2,\ldots)$ be an infinite set of commuting variables,
and for each $p\geq 0$ let $h_p=h_p(x)$ be the $p$-th {\em complete
symmetric function}, that is, formal sum of all monomials in the $x_i$
of degree $p$. The ring $\Lambda = \Z[h_1,h_2,\ldots]$ may be
identified with the ring of symmetric functions in the variables
$x_i$.  There is a ring isomorphism $A\to \Lambda$ sending $u_p$ to
$h_p$ for all $p$.  For any partition $\la$, the element $U_\la$ is
mapped to the Schur $S$-function $s_\la(x)$. The equation
\[
s_{\la}(x) = \det(h_{\la_i+j-i}(x))_{1\leq i,j\leq \ell(\la)}
\]
expressing the Schur functions in terms of the complete symmetric
functions is called the Jacobi-Trudi identity 
(see e.g.\ \cite[I.(3.4)]{M}).

\section{The Hall-Littlewood theory}
\label{HL}

\subsection{The Giambelli formula}

Let $v_1,v_2,\ldots$ be an infinite family of commuting variables,
with $v_i$ having degree $i$ for all $i$. As before let $v_0=1$,
$v_r=0$ for $r<0$, $v_\al= \prod_i v_{\al_i}$ for each integer
sequence $\al$, and $R\, v_\al = v_{R\al}$ for any raising operator
$R$. Let $t$ be a formal variable and consider the graded polynomial
ring $A_t=\Z[t][v_1,v_2,\ldots]$. 

For any integer vector $\al$, define $V_{\al}\in A_t$ by the 
{\em Giambelli formula}
\begin{equation}
\label{giambelliHL}
V_{\al} := \prod_{i<j} \frac{1-R_{ij}}{1-tR_{ij}}\, v_{\al}.
\end{equation}
Let $\al=(\al_1,\ldots,\al_{\ell-1})$ be an arbitrary integer vector
and $r\in \Z$. Expanding the raising operator product in 
(\ref{giambelliHL})
\[  
\prod_{1\leq i<j \leq \ell} \frac{1-R_{ij}}{1-tR_{ij}} = 
\prod_{1\leq i<j < \ell} \frac{1-R_{ij}}{1-t R_{ij}}
\prod_{i=1}^{\ell-1}  \frac{1-R_{i\ell}}{1-t R_{i\ell}}
\]
along the last (i.e., the $\ell$-th) component of $(\al,r)$ gives
\begin{equation}
\label{recurseHL}
V_{(\al,r)} = \sum_\gamma t^{|\gamma|-\#\gamma}(t-1)^{\#\gamma} V_{\al+\gamma}
v_{r-|\gamma|},
\end{equation}
summed over all compositions $\gamma\in \N^{\ell-1}$, where
$\N =\{0,1,\ldots\}$. Equation (\ref{recurseHL}) may be used to give
a recursive definition of $V_{\al}$. The $v_\la$ and $V_\la$ for $\la$ a 
partition form two free $\Z[t]$-bases of $A_t$.

\begin{prop} \label{idHL}
Suppose that we have an equation in $A_t$
\[
\sum_\nu a_\nu V_\nu = \sum_\nu b_\nu V_\nu
\]
where the sums are over all integer vectors $\nu = (\nu_1, \ldots,
\nu_\ell)$, while $a_\nu$ and $b_\nu$ are polynomials in $\Z[t]$, only
finitely many of which are non-zero. Then we have
\[
\sum_\nu a_\nu V_{(\mu,\nu)} = \sum_\nu b_\nu V_{(\mu,\nu)}
\]
for any integer vectors $\mu$.
\end{prop}
\begin{proof}       
It suffices to show that if $\sum_\nu c_\nu V_\nu = 0$ for some 
$c_\nu\in\Z[t]$, then $\sum_\nu c_\nu  V_{(\mu,\nu)} = 0$. 
We will prove that $\sum_\nu c_\nu  V_{(p,\nu)} = 0$ 
for any integer $p$; the desired result then follows by induction.

Upon expanding the raising operators in the definition of $V_\nu$, the
equation $\sum_\nu c_\nu V_\nu = 0$ becomes $\sum_\al \wt{c}_\al v_\al
= 0$ for some coefficients $\wt{c}_\al\in\Z[t]$. We therefore must
show that $\sum_\al \wt{c}_\al \Psi\, v_{(p,\al)} = 0$, where $\dis
\Psi = \prod_{j=2}^{\ell+1} \frac{1-R_{1,j}}{1-tR_{1,j}}$.  For any
integer sequence $\al=(\al_1, \ldots, \al_\ell)$ and every permutation
$\tau$ in the symmetric group $S_\ell$, define
$\tau(\al)=(\al_{\tau(1)},\ldots,\al_{\tau(\ell)})$. By exchanging
rows, we may assume that $\sum_\al \wt{c}_\al v_\al$ and $\sum_\al
\wt{c}_\al \Psi\, v_{(p,\al)}$ are both summed over {\em decreasing}
integer sequences $\al$, because
\[
\Psi \, v_{(p,\al)} = \sum_{\gamma\geq 0} 
t^{|\gamma|-\#\gamma}(t-1)^{\#\gamma} 
v_{p+|\gamma|}v_{\al-\gamma} = \Psi \, v_{(p,\tau(\al))}
\]
for every $\tau\in S_\ell$.  If $\al$ has a negative component then
clearly $v_\al=\Psi\, v_{(p,\al)}=0$.  Moreover, since the $v_\al$ for
$\al$ a partition form a $\Z[t]$-basis of $A_t$, it follows from
$\sum_\al \wt{c}_\al v_\al = 0$ that $\wt{c}_\al=0$ for all partitions
$\al$. We therefore have $\sum_\al \wt{c}_\al \Psi\, v_{(p,\al)} = 0$, as
desired.
\end{proof}

For any integers $r$ and $s$, we claim that the equation
\begin{equation}
\label{comHL}
V_{(r,s)} + V_{(s-1,r+1)} = t\,(V_{(r+1,s-1)} + V_{(s,r)})
\end{equation}
holds in the ring $A_t$. Indeed, (\ref{comHL}) follows from the
identity
\[
\frac{1-R_{12}}{1-tR_{12}} = (1-R_{12}) + t\, 
\frac{1-R_{12}}{1-tR_{12}}\, R_{12}
\]
and the fact that $(1-R_{12}) (v_{(r,s)}+v_{(s-1,r+1)}) = 0$ in $A_t$.

\begin{lemma}
\label{cdcor}
For $c\in \Z$ and $d\geq 1$, we have
\begin{equation}
\label{cdeq}
V_{(c,c+d)}+(1-t)\sum_{i=1}^{d-1}V_{(c+i,c+d-i)} =
t\,V_{(c+d,c)}\,.
\end{equation}
\end{lemma}
\begin{proof}
We use induction on $d$. When $d=1$ the result follows by setting
$s=r+1$ in (\ref{comHL}). If $d>1$, the inductive hypothesis
gives
\[
(1-t)\sum_{i=1}^{d-1}V_{(c+i,c+d-i)} =
V_{(c+d-1,c+1)} - t\, V_{(c+1,c+d-1)}.
\]
The identity (\ref{cdeq}) thus reduces to
\[
V_{(c,c+d)} + V_{(c+d-1,c+1)} - t \,V_{(c+1,c+d-1)} = t\,V_{(c+d,c)},
\]
and this follows directly from equation (\ref{comHL}).
\end{proof}

We can generalize the identity (\ref{comHL}) as follows.

\begin{lemma} 
\label{commuteHL}
Let $\al$ and $\be$ be integer vectors.
Then for any $r,s\in\Z$ we have
\begin{equation}
\label{cHL} 
V_{(\al,r,s,\be)} + V_{(\al,s-1,r+1,\be)} = 
t\,(V_{(\al,r+1,s-1,\be)} + V_{(\al,s,r,\be)})
\end{equation}
in the ring $A_t$.
\end{lemma}
\begin{proof}
By Proposition \ref{idHL}, we may assume that $\al$ is empty. If
$\beta=(\beta',b)$ has positive length, where $b\in \Z$, we set
$\mu=(r,s,\beta')$ and the identity follows by induction, since
\[
V_{(\mu,b)} = 
\sum_\gamma t^{|\gamma|-\#\gamma}(t-1)^{\#\gamma} V_{\mu+\gamma}
v_{b-|\gamma|}.
\]
Finally, if both $\alpha$ and $\beta$ are empty, the result is true by
(\ref{comHL}).
\end{proof}

 Notice that Lemma \ref{commuteA} is the specialization of Lemma
\ref{commuteHL} at $t=0$.  Using Lemma \ref{commuteHL} with
$\al=\emptyset$ and arguing as in the proof of Lemma \ref{cdcor} gives
the following result.

\begin{cor}
\label{cdcorab}
For $c\in \Z$, $d\geq 1$, and $\be$ any integer vector, 
we have
\[
V_{(c,c+d,\be)}+(1-t)\sum_{i=1}^{d-1}V_{(c+i,c+d-i,\be)} = 
t\,V_{(c+d,c,\be)}\,.
\]
\end{cor}

\subsection{The Pieri rule} 
\label{pieriruleHL}

Let $\la\subset\mu$ be two partitions such that $\mu/\la$ is a
horizontal strip, and let $J$ be the set of integers $c\geq 1$ such
that $\mu/\la$ does not (respectively does) have a box in column $c$
(respectively column $c+1$). Define
\[
\psi_{\mu/\la}(t) = \prod_{c\in J}(1-t^{m_c(\la)}),
\]
where $m_c(\la)$ denotes the number of parts of $\la$ that are 
equal to $c$. 

For any $d\geq 1$ define the raising operator $R_t^d$ by
\[
R_t^d = \prod_{1\leq i<j \leq d} \frac{1-R_{ij}}{1-tR_{ij}}.
\]
Given a partition $\la$ of length $\ell$, we compute 
\[
v_p\cdot V_\la = v_p\cdot R_t^{\ell}\, v_{\la} = R_t^{\ell+1}\cdot
\prod_{i=1}^\ell\frac{1-tR_{i,\ell+1}}{1-R_{i,\ell+1}} \, v_{(\la,p)}
\]
\[
= R_t^{\ell+1}\cdot\prod_{i=1}^\ell(1+(1-t)R_{i,\ell+1} + 
(1-t)R_{i,\ell+1}^2 + \cdots)\,v_{(\la,p)} 
\]
and therefore
\begin{equation}
\label{1steq}
v_p\cdot V_\la = \sum_{\nu\in \cN} (1-t)^{\#(\nu-\la)}\,V_\nu,
\end{equation}
where $\cN=\cN(\la,p)$ is the set of compositions $\nu\geq \la$
such that $|\nu| = |\la|+p$ and $\nu_j =0$ for
$j > \ell+1$, as in \S \ref{pieriruleA}.

For any integer composition $\nu$, set 
$\nu^* = (\nu_2,\nu_3,\ldots)$. Define
\[
\cN^*=\{\nu\in\cN\ |\ \nu_j\leq\la_{j-1} \text{ for } j>2\}
\]
and for $\nu\in \cN^*$, let
\[
T_{\nu} = (1-t)^{\#(\nu_1-\la_1)}\psi_{\nu^*/\la^*}(t)\,V_\nu \,.
\] 
Using Proposition \ref{idHL} and induction on the length of $\la$
gives
\begin{equation}
\label{2ndeq}
\sum_{\nu\in \cN} (1-t)^{\#(\nu-\la)}\,V_\nu =
\sum_{\nu\in\cN^*}T_{\nu}.  
\end{equation}
Note that if $\nu\in \cN^*$ satisfies
$\nu_2<\la_1$, then $T_{\nu} = \psi_{\nu/\la}(t) \, V_\nu$.
Therefore
\begin{equation}
\label{intereq}
\sum_{\nu\in\cN^*}T_{\nu} = \sum_{{\nu\in \cN^*}\atop{\nu_2<\la_1}}
\psi_{\nu/\la}(t) \, V_\nu + \sum_{{\nu\in
\cN^*}\atop{\nu_2\geq\la_1}} T_\nu.
\end{equation}
We claim that for each fixed $d\geq 0$,
\begin{equation}
\label{star2}
\sum_{{\nu\in \cN^*}\atop{{\nu_2\geq \la_1}
\atop{\nu_1+\nu_2=2\la_1+d}}} T_\nu \ = \ 
\sum_{{\nu\in\cN^*}\atop{{\nu_2=\la_1}\atop{\nu_1 = \la_1+d}}} 
\psi_{\nu/\la}(t) \, V_\nu.
\end{equation}
Equation (\ref{star2}) is justified by using Corollary \ref{cdcorab}
to fix all rows except the first two. We then apply the identity
\begin{equation}
\label{basic1}
V_{(c,c+d)}+(1-t)\sum_{i=1}^{d-1}V_{(c+i,c+d-i)} + (1-t)\,V_{(c+d,c)} 
= V_{(c+d,c)} 
\end{equation}
when $\nu\ssm\la$ has a box in column $c=\la_1$, and the identity
\begin{equation}
\label{basic2}
(1-t^m)\,V_{(c,c+d)}+(1-t)(1-t^m)\sum_{i=1}^{d-1}V_{(c+i,c+d-i)} 
+ (1-t)\,V_{(c+d,c)} = (1-t^{m+1})\,V_{(c+d,c)}
\end{equation}
for $m\geq 1$, otherwise. The identities (\ref{basic1}) and
(\ref{basic2}) follow directly from Lemma \ref{cdcor}, proving the
claim. 

Summing (\ref{star2}) over all $d\geq 0$, applying the result to
(\ref{intereq}), and taking (\ref{1steq}) and (\ref{2ndeq}) into
account, we finally obtain the {\em Pieri rule}
\begin{equation}
\label{pieriHL}
v_p\cdot V_\la = \sum_{\mu} \psi_{\mu/\la}(t) \, V_\mu,
\end{equation}
summed over all partitions $\mu\supset\la$ with
$|\mu| = |\la|+p$ and $\mu/\la$ a horizontal $p$-strip.

\medskip

Conversely, suppose that we are given a family $\{T_\la\}$ of elements
of $A_t$, one for each partition $\la$, such that $T_p=v_p$ for every
integer $p\geq 0$ and the $T_\la$ satisfy the Pieri rule $T_p\cdot
T_\la = \sum_{\mu} \psi_{\mu/\la}(t) \, T_\mu$, with the sum over
$\mu$ and $\psi_{\mu/\la}(t)$ as in (\ref{pieriHL}).  It follows then
as in \S \ref{pieriruleA} that
\[
T_\la = V_\la = \prod_{i<j}\frac{1-R_{ij}}{1-tR_{ij}}\, v_\la.
\] 
We have established that the Giambelli formula (\ref{giambelliHL}) 
and Pieri rule (\ref{pieriHL}) are formally equivalent to each other.

\subsection{Hall-Littlewood symmetric functions}
\label{oddsendsHL}

Let $x=(x_1,x_2,\ldots)$ be as in \S \ref{oddsendsA}, define the formal 
power series $q_r(x\,;t)$ by the generating equation
\[
\prod_{i=1}^{\infty}\frac{1-x_itz}{1-x_iz} = \sum_{r=0}^{\infty}q_r(x\,;t)z^r
\]
and set $\Gamma_t = \Z[t][q_1,q_2,\ldots]$. There is a $\Z[t]$-linear
ring isomorphism $A_t\to\Gamma_t$ sending $v_r$ to $q_r(x\,;t)$ for all
$r\geq 1$. For any partition $\la$, the element $V_\la$ is mapped to the
{\em Hall-Littlewood function} $Q_\la(x\,;t)$. The raising operator
formula (\ref{giambelliHL}) for Hall-Littlewood functions is due to
Littlewood \cite{Li}, who also obtained Lemma \ref{commuteHL} in this
setting. The Pieri rule (\ref{pieriHL}) for the functions $Q_\la(x\,;t)$
was first proved by Morris \cite{Mo}; an alternative proof may be
found in \cite[III.5]{M}.

If $\Lambda$ denotes the ring of symmetric functions from \S
\ref{oddsendsA}, then the $\Z[t]$-linear ring isomorphism
$A_t\to\Lambda[t]$ sending $v_r$ to $h_r(x)$ for every $r\geq 1$ maps
$V_\la$ to the {\em modified Hall-Littlewood function} $Q_\la'(x\,;t)$
of Lascoux, Leclerc, and Thibon \cite{LLT}. Now specialize $t=-1$ in
the definition of the ring $A_t$ and its distinguished basis $V_\la$,
and consider the isomorphism $A_{-1}\to\Lambda$ obtained by mapping
$v_r$ to the $r$-th elementary symmetric function $e_r(x)$ for every
$r\geq 1$. The basis element $V_\la$ is then mapped to the
$\wt{Q}$-polynomial $\wt{Q}_\la(x)$ of Pragacz and Ratajski \cite{PR}.

\section{The type C theory: Isotropic Grassmannians}
\label{typeC:general}

Fix an integer $k\geq 0$.  The aim of this section is to explain the
main conclusions of \cite{BKT1, BKT2} regarding the Giambelli
and Pieri formulas for the Grassmannian $\IG=\IG(n-k,2n)$ of isotropic
$(n-k)$-dimensional subspaces of $\C^{2n}$, equipped with a symplectic
form. We provide an exposition of those aspects relevant to the
present article, including examples, and refer the reader to the
original papers for detailed proofs.

\subsection{The Giambelli formula} 
We consider an infinite family $w_1,w_2,\ldots$ of commuting variables,
with $w_i$ of degree $i$ for all $i$, and set $w_0=1$, $w_r=0$ for $r<0$, 
and $w_\al= \prod_i w_{\al_i}$ as before. 
Let $I^{(k)}\subset \Z[w_1,w_2,\ldots]$ be the ideal generated by the
relations
\begin{equation}
\label{kpresrels}
\frac{1-R_{12}}{1+R_{12}}\, w_{(r,r)} = 
w_r^2 + 2\sum_{i=1}^r(-1)^i w_{r+i}w_{r-i}= 0
\ \ \ \text{for} \  r > k.
\end{equation}
Define the graded ring $B^{(k)}=\Z[w_1,w_2,\ldots]/{I^{(k)}}$. All
equations in $B^{(k)}$, such as the Giambelli and Pieri formulas, will
be valid only up to the elements of $I^{(k)}$.

Let $\Delta^{\circ} = \{(i,j) \in \N \times \N \mid 1\leq i<j \}$,
equipped with the partial order $\leq$ defined by $(i',j')\leq (i,j)$
if and only if $i'\leq i$ and $j'\leq j$.  A finite subset $D$ of
$\Delta^{\circ}$ is a {\em valid set of pairs} if it is an order
ideal, i.e., $(i,j)\in D$ implies $(i',j')\in D$ for all $(i',j')\in
\Delta^{\circ}$ with $(i',j') \leq (i,j)$.

A partition $\la$ is {\em $k$-strict} if all its parts greater than $k$
are distinct. Given a $k$-strict partition $\la$, we define a set
of pairs $\cC(\la)$ by
\[
\cC(\la) = \{(i,j)\in\Delta^{\circ}\ |\ \la_i+\la_j > 2k+j-i \ \,
\text{and} \ \, j \leq \ell(\la)\}.
\]
It is easy to check that $\cC(\la)$ is a valid set of pairs.
Conversely, assuming $k>0$, let $D$ be any 
valid set of pairs $D$ and set $d_i=\#\{j\ |\ (i,j)\in D\}$. Then the 
prescription
\[
\la_i = 
\begin{cases}
k+1+d_i & \text{if $d_i>0$}, \\
k & \text{if $d_i=0$}
\end{cases}
\]
for $1\leq i \leq d_1+1$ defines a $k$-strict partition $\la$ such
that $\cC(\la)= D$.  The white dots in Figure \ref{Csetandrim}
illustrate a typical valid set of pairs.

\begin{figure}
\centering
\includegraphics[scale=0.3,viewport=0 0 700 240]{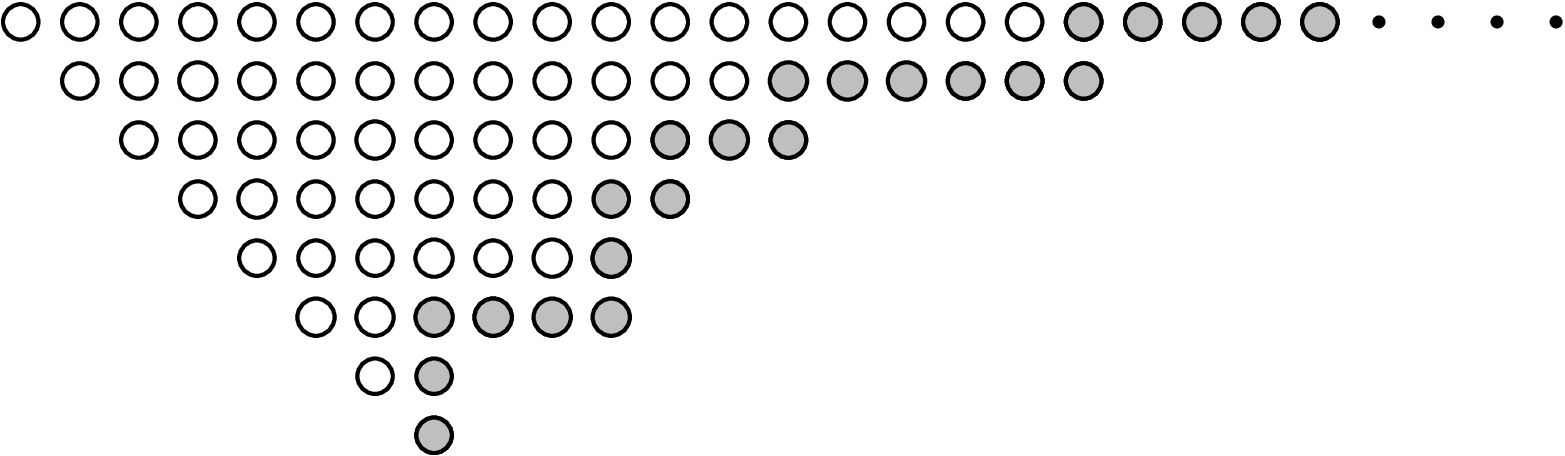}
\caption{A set of pairs (in white) and its outside
rim (in grey).}
\label{Csetandrim}
\end{figure}

For any valid set of pairs $D$, we define the raising operator
\[
R^D = \prod_{i<j}(1-R_{ij})\prod_{i<j\, :\, (i,j)\in D}(1+R_{ij})^{-1}.
\]
If $\la$ is a $k$-strict partition, then set $R^\la := R^{\cC(\la)}$. 
For any such $\la$, define $W_{\la}\in B^{(k)}$ by the 
{\em Giambelli formula}
\begin{equation}
\label{giambelliCgen}
W_{\la} := R^\la\, w_{\la}.
\end{equation}
In contrast to (\ref{giambelliA}) and (\ref{giambelliHL}), we see that
the raising operator $R^{\la}$ in the Giambelli definition
(\ref{giambelliCgen}) depends on $\la$.
When $\la_i+\la_j\leq 2k+j-i$ for all $i<j$, equation
(\ref{giambelliCgen}) becomes $W_\la = \prod_{i<j}(1-R_{ij})\, w_\la$,
while when $\la_i+\la_j > 2k+j-i$ for $1\leq i<j\leq \ell(\la)$, the
formula is equivalent to $\dis W_\la =
\prod_{i<j}\frac{1-R_{ij}}{1+R_{ij}}\, w_\la$.

\begin{example} In the ring $B^{(2)}$ we have
\begin{align*}
W_{621} 
&= \frac{1-R_{12}}{1+R_{12}}\frac{1-R_{13}}{1+R_{13}}(1-R_{23})\, w_{621} \\
&= (1-2R_{12}+2R_{12}^2 - \cdots)
(1-2R_{13}+2R_{13}^2 - \cdots)
(1-R_{23}) \, w_{621} \\
&= (1-2R_{12}+2R_{12}^2 - 2 R_{12}^3)(1-2R_{13}-R_{23})\, w_{621} \\
&= w_{621}-w_{63}-2w_{711}+4w_{81}-2w_{9} \\
&= w_6w_2w_1 -w_6w_3 -2w_7w_1^2 +4w_8w_1 -2w_9.
\end{align*}
\end{example}

From equation (\ref{kpresrels}) we deduce that either the partition
$\la$ is $k$-strict, or $w_\la$ is a $\Z$-linear combination of the
$w_\mu$ such that $\mu$ is $k$-strict and $\mu \succ \la$. It follows
that the $w_\la$ for $\la$ $k$-strict span $B^{(k)}$ as an abelian
group. A dimension counting argument shows that in fact, the $w_\la$
for $\la$ $k$-strict form a $\Z$-basis of $B^{(k)}$.  From the
definition (\ref{giambelliHL}) it follows that for any partition
$\la$, $W_\la$ is of the form
\[
W_\la = w_\la + \sum_{\mu\succ\la} a_{\la\mu}\, w_\mu
\]
with coefficients $a_{\la\mu}\in\Z$. We deduce that 
\[
W_\la = w_\la + \sum_{\mu\succ\la} b_{\la\mu} \, w_\mu
\]
with the sum restricted to $k$-strict partitions $\mu\succ\la$. Therefore,
the $W_\la$ as $\la$ runs over $k$-strict partitions form a $\Z$-basis of
$B^{(k)}$.

More generally, given a valid set of pairs $D$ and an integer sequence
$\alpha$, we denote $R^D\, w_{\al}$ by $W^D_\al$. We say that $D$ is
the {\em denominator set} of $W^D_\al$.  If $r,s\in
\Z$, then
\[
W^{\emptyset}_{(r,s)} = - W^{\emptyset}_{(s-1,r+1)}
\]
while if $D\neq \emptyset$ and $r+s>2k$, then
\[
W^D_{(r,s)} = - W^D_{(s,r)}
\]
in the ring $B^{(k)}$. To generalize these equations, care is
required, as e.g.\ the direct analogue of Lemma \ref{commuteA} for the
$W^D_\al$ fails. Suppose $D$ is a valid set of pairs and $j\geq 1$.
We say that the pair $(j,j+1)$ is {\em $D$-tame} if

\medskip
(i) $(j,j+1)\notin D$ and for all
$h<j$, $(h,j)\notin D$ if and only if $(h,j+1)\notin D$, or 

\medskip
(ii)
$(j,j+1)\in D$ and for all $h>j+1$, $(j,h)\in D$ if and only if
$(j+1,h)\in D$.

\begin{lemma}[\cite{BKT2}]
\label{commuteAC}
Let $\al=(\al_1,\ldots,\al_{j-1})$ and
$\be=(\be_{j+2},\ldots,\be_\ell)$ be integer vectors, and assume that
$(j,j+1)$ is $D$-tame. 

\medskip
\noin
{\em (a)} If $(j,j+1)\notin D$, then 
for any $r,s\in\Z$ we have
\[ W^D_{(\al,r,s,\be)} = - W^D_{(\al,s-1,r+1,\be)} \,. \]

\medskip
\noin
{\em (b)} If $(j,j+1)\in D$, then for any
$r,s\in \Z$ such that $r+s > 2k$, we have
\[ W^D_{(\al,r,s,\be)} = - W^D_{(\al,s,r,\be)} \,. \]
\end{lemma}

\begin{example}
The ring $B:=B^{(0)}$ has free $\Z$-basis consisting of the $W_\la$ 
for $\la$ a strict partition, and the Giambelli formula 
(\ref{giambelliCgen}) is equivalent to the identity
\begin{equation}
\label{giambelliC}
W_\la = \prod_{i<j}\frac{1-R_{ij}}{1+R_{ij}}\, w_\la.
\end{equation}
Lemma \ref{commuteAC} in this case gives $W_{(\al,r,s,\be)} =
-W_{(\al,s,r,\be)}$ whenever $r+s > 0$.
\end{example}

\subsection{The Pieri rule} 
\label{pierirulegen}
We begin our analysis in the same way as in the previous sections.
For $p>0$ and any $k$-strict partition $\la$ of length $\ell$, we find 
that
\begin{equation}
\label{starteq}
w_p\cdot W_\la = \sum_{\nu\in \cN} W^{\cC(\la)}_\nu,
\end{equation}
where $\cN=\cN(\la,p)$ is the set of all compositions $\nu\geq \la$
such that $|\nu| = |\la|+p$ and $\nu_j =0$ for $j > \ell+1$. The
problem with the right hand side of (\ref{starteq}) is two-fold:
first, the compositions $\nu\in\cN$ are not $k$-strict partitions, and
second, the denominator set $\cC(\la)$ is the same for each term in
the sum, and will have to be modified so as to agree with the summands
in the eventual Pieri rule. Rather than give the complete argument,
we will state the Pieri rule which results and then discuss some
features of the proof.

We let $[r,c]$ denote the box in row $r$ and column $c$ of a Young
diagram. Suppose that $c\leq k<c'$. We say that the boxes $[r,c]$ and
$[r',c']$ are {\em $k$-related\/} if $c+c'=2k+2+r-r'$.  In the diagram
of Figure \ref{krelated}, the two grey boxes are $k$-related.
\begin{figure}
\centering
\pic{0.65}{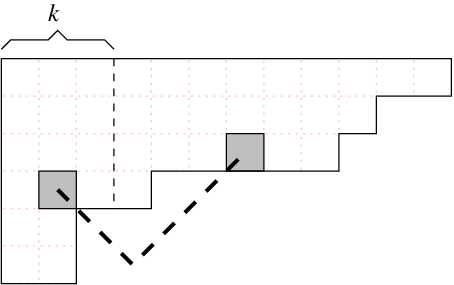} 
\caption{Two $k$-related boxes in a $k$-strict Young diagram}
\label{krelated}
\end{figure}

Given two partitions $\la$ and $\mu$ with $\la\subset\mu$, the skew
Young diagram $\mu/\la$ is called a {\em vertical strip} if it does
not contain two boxes in the same row. For any two $k$-strict
partitions $\la$ and $\mu$, let $\al_i$ (respectively $\be_i$) denote
the number of boxes of $\la$ (respectively $\mu$) in column $i$, for
$1\leq i\leq k$.  We have a relation $\lambda \to \mu$ if $\mu$ can be
obtained by removing a vertical strip from the first $k$ columns of
$\lambda$ and adding a horizontal strip to the result, so that for
each $i$ with $1\leq i\leq k$,

\medskip
\noin
(1) if $\be_i=\al_i$, then the box $[\al_i,i]$ 
is $k$-related to at most one box of $\mu \smallsetminus \lambda$; and

\medskip
\noin (2) if $\be_i < \al_i$, then the boxes
$[\be_i,i],\ldots,[\al_i,i]$ must each be $k$-related to
exactly one box of $\mu \smallsetminus \lambda$, and these boxes of
$\mu \smallsetminus \lambda$ must all lie in the same row.

\medskip
If $\lambda \to \mu$, we let $\A$ be the set of boxes of $\mu\ssm \la$
in columns $k+1$ and higher which are {\em not} mentioned in (1) or
(2). Define the connected components of $\A$ by agreeing that two
boxes in $\A$ are connected if they share at least a vertex. Then
define $N(\lambda,\mu)$ to be the number of connected components of
$\A$ which do not have a box in column $k+1$.  Finally, we can state
the Pieri rule for $B^{(k)}$: Given any $k$-strict partition $\lambda$
and integer $p\geq 0$,
\begin{equation}
\label{finaleq}
 w_p \cdot W_\lambda = \sum_{{\lambda \to \mu} \atop {|\mu|=|\lambda|+p}} 
 2^{N(\lambda,\mu)} \, W_\mu \,.
\end{equation}

The {\em outside rim} $\partial\cC$ of $\cC=\cC(\la)$ is the set of
pairs $(i,j)\in\Delta^{\circ}\ssm\cC$ such that $i=1$ or $(i-1,j-1) \in
\cC$; an example of these sets is displayed in Figure
\ref{Csetandrim}. The Young diagrams of the partitions $\mu$ which
appear in (\ref{finaleq}) do not all contain the diagram of $\la$, in
contrast to the Pieri rule of \S \ref{pieriruleHL} for Hall-Littlewood
functions. However, the denominator sets $\cC(\mu)$ for the partitions
$\mu$ with $\la\to\mu$ all contain $\cC$, are contained in
$\cC\cup\partial\cC$, and are empty beyond column $\ell+1$.

To prove (\ref{finaleq}), our task is to show that 
\begin{equation}
\label{toshow}
\sum_{\nu\in \cN(\la,p)} W^{\cC(\la)}_\nu = \sum_{{\lambda \to \mu} \atop
  {|\mu|=|\lambda|+p}} 2^{N(\lambda,\mu)} \, W^{\cC(\mu)}_\mu \,.
\end{equation}
It is clear that we need a mechanism to modify the denominator sets of
the terms $W^{\cC(\la)}_\nu$ on the left hand side of
(\ref{toshow}). This is based on the observation that if $D$ is a
denominator set and $(i,j)\in \Delta^{\circ}\smallsetminus D$, then
\begin{equation}
\label{mitosis}
W_{\al}^D = W_{\al}^{D\cup(i,j)} + W_{R_{ij}\al}^{D\cup (i,j)}\,.
\end{equation}
Equation (\ref{mitosis}) follows directly from the raising operator
identity
\[
1-R_{ij} = \frac{1-R_{ij}}{1+R_{ij}} + \frac{1-R_{ij}}{1+R_{ij}}\,R_{ij}.
\qedhere
\]

The following three detailed examples illustrate how Lemma \ref{commuteAC} 
and (\ref{mitosis}) may be used repeatedly to obtain (\ref{toshow}). For
simplicity, the commas are omitted from the notation for integer vectors
and pairs.

\begin{example}
\label{ex1}
The following chain of equalities holds in $B^{(1)}$.
\begin{align*}
w_1\cdot W_{3211} 
&= W_{32111}^{12} + W_{3212}^{12} + W_{3221}^{12} + W_{3311}^{12}
+ W_{4211}^{12} \\
&= W_{32111}^{12} + \left(W_{3221}^{12,13} + W_{4211}^{12,13}\right)
+ \left(W_{4211}^{12,13} + W_{5201}^{12,13}\right) \\
&= W_{32111}^{12} + (W_{3221}^{12,13,23} + W_{3311}^{12,13,23})
 + 2\, W_{4211}^{12,13} + (W_{5201}^{12,13,14}+ 
W_{6200}^{12,13,14}) \\
&= W_{32111}^{\cC(32111)} + 2\, W_{4211}^{\cC(4211)} + W_{62}^{\cC(62)}.
\end{align*}
Observe that Lemma \ref{commuteAC}(a) is used to show that 
$W_{3212}^{12} = W_{5201}^{12,13,14}=0$, 
while Lemma \ref{commuteAC}(b) implies that 
$W_{3311}^{12} = W_{3221}^{12,13,23} = W_{3311}^{12,13,23}=0$.
\end{example}

\begin{example}
\label{ex2}
The following chain of equalities holds in $B^{(1)}$.
\begin{align*}
w_3\cdot W_{21} 
&= W_{51}^{\emptyset} + W_{42}^{\emptyset}+
W_{411}^{\emptyset} + W_{33}^{\emptyset}+W_{24}^{\emptyset} \\
&\qquad + W_{321}^{\emptyset} + W_{312}^{\emptyset}+
W_{231}^{\emptyset} + W_{222}^{\emptyset}+
W_{213}^{\emptyset} \\
&=\left(W_{51}^{12}+W_6^{12}\right) + 
\left(W_{42}^{12}+ W_{51}^{12}\right) + 
\left(W_{411}^{12}+W_{501}^{12}\right) +
\left(W_{33}^{12}+W_{42}^{12}\right) \\
&\qquad +\left(W_{24}^{12}+W_{33}^{12}\right) 
+\left(W_{321}^{12}+ W_{411}^{12}\right) 
+\left(W_{231}^{12}+W_{321}^{12}\right) \\
&=W_6^{12}+2\,W_{51}^{12}+W_{42}^{12}+
\left(W_{411}^{12,13}+W_{51}^{12,13} \right) +
\left(W_{501}^{12,13}+W_6^{12,13}\right) \\
&\qquad +W_{321}^{12}+\left(W_{411}^{12,13}+W_{51}^{12,13}\right) \\
&=2\, W_6^{\cC(6)}+4\,W_{51}^{\cC(51)}+W_{42}^{\cC(42)} + 
2\, W_{411}^{\cC(411)} + W_{321}^{\cC(321)}.
\end{align*}
Here we use Lemma \ref{commuteAC}(a) to see that 
\[
 W_{312}^{\emptyset} = W_{222}^{\emptyset}+W_{213}^{\emptyset}
= W_{501}^{12,13} = 0,
\]
while Lemma \ref{commuteAC}(b) gives 
\[
W_{33}^{12} = W_{42}^{12} + W_{24}^{12} = W_{231}^{12}+W_{321}^{12} = 0.
\]
\end{example}

\begin{example}
\label{ex0}
Let $k=0$. We give a complete cancellation scheme along the 
above lines assuming that {\em the integer $p$ is less than or 
equal to $\la_\ell$}. For any integers $d,e\geq 1$ define the raising 
operator $R^{[d,e]}$ by
\[
R^{[d,e]} = \prod_{1\leq i<j \leq d} (1-R_{ij})
\prod_{1\leq i<j \leq e} (1+R_{ij})^{-1}.
\]
We compute that
\[
w_p\cdot W_\la = R^{[\ell,\ell]} \, w_{(\la,p)} =
\sum_{\nu\in\cN}R^{[\ell+1,\ell]}\,w_\nu
= \sum_{\nu\in\cN}R^{[\ell+1,\ell+1]}\,\prod_{i=1}^\ell(1+R_{i,\ell+1})
\, w_\nu
\]
and therefore
\begin{equation}
\label{initeqC}
w_p\cdot W_\la = \sum_{(\nu,\gamma)\in\cN'(\la,p)} W_{\nu+\gamma}
\end{equation}
where $\cN'(\la,p)$ denotes the set of all pairs $(\nu,\gamma)$ of
integer vectors of length at most $\ell+1$ with $\nu\in\cN$,
$\nu+\gamma\geq \la$, $|\nu+\gamma| = |\la|+p$, and
$\gamma_i\in\{0,1\}$ for $1\leq i\leq \ell$.

For every pair $(\nu,\gamma)\in \cN'(\la,p)$, define 
$\mu= \nu+\gamma$. Call a pair $(\nu,\gamma)$ 
{\em bad} if there exists a $j>1$ such that 
\begin{center}
(i) $\mu_j = \mu_{j-1}$ \ or \
(ii) $\mu_j > \la_{j-1}$
\ or \ (iii) $\mu_j = \la_{j-1}$ and $\gamma_j = 0$. 
\end{center}
Let $X$ be the set of all bad pairs in $\cN'(\la,p)$. We
next define an involution $\iota:X\to X$ as follows. For
$(\nu,\gamma)\in X$, choose $j$ maximal such that (i), (ii), or (iii)
holds. If $\mu_j= \mu_{j-1}$, we let $(\nu',\gamma') =
(\nu,\gamma)$. Otherwise, let $\varpi\in S_\ell$ be the transposition
$(j-1,j)$ and define $(\nu',\gamma')$ by setting $(\nu'_i,\gamma'_i)=
(\nu_{\varpi(i)},\gamma_{\varpi(i)})$ for $1\leq i\leq \ell$. Finally,
set $\iota(\nu,\gamma) = (\nu',\gamma')$ for each $(\nu,\gamma)\in X$.
Lemma \ref{commuteAC} applied to rows $j-1$ and $j$ gives
$W_{\nu+\gamma} = -W_{\nu'+\gamma'}$ for every $(\nu,\gamma)\in X$,
therefore all bad indices may be omitted from the sum in
(\ref{initeqC}).  

We are left with pairs $(\nu,\gamma)$ such that
$\mu=\nu+\gamma\supset\la$ is a strict partition with $\mu/\la$ a
horizontal strip. Observe that every connected component $C$ of $\mu/\la$
which does not lie in row $\ell+1$ contributes a multiplicity of $2$
to the sum (\ref{initeqC}). Indeed, if $C$ lies in rows $r$ through
$s$ with $r\leq s$, then condition (iii) implies that $\gamma_i=1$ for
$r<i\leq s$, while $\gamma_r$ can be either $0$ or $1$. We have
therefore proved the Pieri rule
\[
w_p\cdot W_\la = \sum_\mu 2^{N(\la,\mu)} \, W_\mu
\]
summed over all strict partitions $\mu\supset\la$ with $|\mu| =
|\la|+p$ and $\mu/\la$ a horizontal $p$-strip. The exponent
$N(\la,\mu)$ equals the number of connected components of
$\mu/\la$ which do not meet the first column. 
\end{example}

\medskip

In the general case, there is a substitution algorithm by which the
left hand side of (\ref{toshow}) {\em evolves} into the right hand
side. Although Examples \ref{ex1}, \ref{ex2}, and \ref{ex0} appear
encouraging, they are still far from a precise formulation of this
algorithm for a general Pieri product $w_p \cdot W_\la$.  It is
instructive to try to prove along these lines that the Pieri rule of
Example \ref{ex0} holds without the simplifying assumption that $p
\leq \la_\ell$. The complete proof for arbitrary $k\geq 0$ is given in
\cite{BKT2}; we note here that for technical reasons, the argument is
performed in type B, that is, for odd orthogonal Grassmannians. Each
summand $W^{\cC(\la)}_\nu$ for $\nu\in\cN$ in (\ref{toshow}) gives
rise to a {\em tree} of successor terms, with the branching given by
repeated substitutions using (\ref{mitosis}).  In this way, we obtain
the {\em substitution forest}; the roots of the trees in the forest
are the $W^{\cC(\la)}_\nu$ for $\nu\in\cN$. By construction, the sum
of all the roots is equal to the sum of all the leaves. When the
algorithm terminates, the set of leaves contains the terms
$W^{\cC(\mu)}_\mu$ for $\la\to\mu$, each appearing exactly
$2^{N(\lambda,\mu)}$ times; the remaining leaves are either zero or
cancel in pairs, as dictated by Lemma \ref{commuteAC}.

Once we know that the Giambelli formula (\ref{giambelliCgen}) used to
define $W_\la$ satisfies the Pieri rule (\ref{finaleq}), the same
reasoning as in the previous two sections establishes that the two
results are formally equivalent.

\subsection{The isotropic Grassmannian and theta polynomials}
\label{oddsendsCgen}

Let the vector space $E=\C^{2n}$ be equipped with a nondegenerate
skew-symmetric bilinear form. A subspace $\Sigma$ of $E$ is {\em
isotropic} if the form vanishes when restricted to $\Sigma$. The
dimensions of such isotropic subspaces $\Sigma$ range from $0$ to $n$;
when $\dim(\Sigma)=n$ we say that $\Sigma$ is a {\em Lagrangian}
subspace.

Choose $n \geq k\geq 0$ and let $\IG=\IG(n-k,2n)$ denote the
Grassmannian parametrizing isotropic subspaces of $E$.  Its cohomology
ring $\HH^*(\IG,\Z)$ has a free $\Z$-basis of Schubert classes
$\s_\la$, one for each $k$-strict partition $\la$ whose diagram is
contained in the $(n-k)\times (n+k)$ rectangle $\cR(n-k,n+k)$.
Following \cite{BKT1, BKT2}, the special Schubert classes $\s_p$ for
$\IG$ are the Chern classes of the universal quotient bundle over
$\IG$, as in \S \ref{oddsendsA}.  There is a ring epimorphism $\psi:
B^{(k)} \to \HH^*(\IG,\Z)$ sending the generators $w_p$ to the special
Schubert classes $\s_p$ for $1\leq p \leq n+k$ and to zero for $p >
n+k$. For any $k$-strict partition $\la$, we have $\psi(W_\la) =
\sigma_\la$ if $\la \subset \cR(n-k,n+k)$, and $\psi(W_\la) = 0$
otherwise.

Let $x=(x_1,x_2,\ldots)$ as in \S \ref{oddsendsA} and set 
$y=(y_1,\ldots,y_k)$. Define the formal power series
$\ti_r(x\,;y)$ by the generating equation
\[
\prod_{i=1}^{\infty}\frac{1+x_it}{1-x_it} \prod_{j=1}^k
(1+y_jt)= \sum_{r=0}^{\infty}\ti_r(x\,;y)t^r.
\]
Set $\Gamma^{(k)} = \Z[\ti_1,\ti_2,\ldots]$ to be the ring of theta
polynomials. There is a ring isomorphism $B^{(k)}\to\Gamma^{(k)}$
sending $w_p$ to $\ti_p$ for all $p$.  For any $k$-strict partition
$\la$, the element $W_\la$ is mapped to the {\em theta polynomial}
$\Ti_{\la}(x\,;y)$ of \cite{BKT2}. These polynomials agree with the
Schubert polynomials of type C defined by Billey and Haiman \cite{BH}
indexed by a Grassmannian permutation of the hyperoctahedral group;
see \cite[\S 6]{BKT2} and \S \ref{sss} of the present paper for
further information.

We next discuss this theory when $k=0$. If $\al$ has length $\ell$ and
$m>0$ is the least integer such that $2m\geq \ell$, then equation
(\ref{giambelliC}) may be written in its original form
\begin{equation}
\label{giambelliC2}
W_{\al} = \Pf(W_{\al_i,\al_j})_{1\leq i<j \leq 2m}.
\end{equation}
To see this, one may argue as in \S \ref{oddsendsA}, this time using
Schur's identity \cite[\S IX]{S}
\[
\prod_{1\leq r<s\leq 2m}\frac{x_r-x_s}{x_r+x_s} = 
\Pf\left(\frac{x_r-x_s}{x_r+x_s}\right)_{1\leq r,s \leq 2m}.
\]
In the theory of symmetric functions, the ring $\Gamma:=\Gamma^{(0)}$
is called the ring of Schur $Q$-functions.  For any strict partition
$\la$, the element $W_\la$ is mapped to the Schur $Q$-function
$Q_\la(x)$, which Schur \cite{S} defined using the Pfaffian equation
(\ref{giambelliC2}). 

Let $\phi$ denote the ring homomorphism $A_{-1} \to \Gamma$ sending
$v_p$ to $q_p(x)$ for every $p\geq 1$. Then $\phi(V_{\la}) = Q_\la(x)$
whenever $\la$ is {\em strict}, and $\phi(V_{\la}) = 0$ for non-strict
partitions $\la$.  The Pieri rule for the products $q_p\cdot Q_{\la}$
may therefore be obtained by specializing Morris' rule (\ref{pieriHL})
for the Hall-Littlewood symmetric functions when $t=-1$. The Pieri and
Giambelli formulas for the Lagrangian Grassmannian $\LG(n,2n)$ were
proved by Hiller and Boe \cite{HB} and Pragacz \cite{Pra},
respectively.

\section{Mirror identities and recursion formulas}
\label{mirrorsec}

\subsection{}
\label{phipsi}

Let $\la$ be any partition of length $\ell$. 
Our proof of the Pieri rule for $V_\la$ in \S \ref{pieriruleHL} 
began with the equation
\[
v_p\cdot V_\la = v_p\cdot R\, v_{\la} = R\, v_{(\la,p)}
\]
for an appropriate raising operator $R$, and relied on the identity
\begin{equation}
\label{straightid}
 \sum_{\al\geq 0} (1-t)^{\#\al}\,V_{\la+\al} = 
\sum_{\mu}\psi_{\mu/\la}(t) \, V_\mu
\end{equation}
where the first sum is over all compositions $\al$ of length at 
most $\ell+1$, and the second over partitions
$\mu\supset\la$ such that $\mu/\la$ is a horizonal strip. 
On the other hand, we may write $v_p\cdot R\, v_{\la} =
R'\,v_{(p,\la)}$ for a raising operator $R'$ which agrees with $R$ 
up to a shift of its indices, and attempt a similar
analysis. When $p$ is sufficiently large, we are led to  
the {\em mirror identity} to (\ref{straightid}), which is an
analogous formula for
the sum $\dis \sum_{\al\geq 0}(1-t)^{\#\al}\,V_{\la-\al}$.

We will require an auxiliary result which stems from \cite{BKT3}.

\begin{lemma}
\label{prodlemmaHL}
Let $P_r$ be the set of partitions $\mu$ with $|\mu| = r$, and 
let $m$ be a positive integer. Then the $\Z[t]$-linear map
\[
\phi\ :\ \bigoplus_{r=0}^{\lfloor\frac{m-1}{2}\rfloor}
\bigoplus_{\mu\in P_r}\Z[t] \to A_t
\]
which, for given $r$ and $\mu\in P_r$, sends the corresponding 
basis element to $v_{m-r}\, V_\mu$, is injective.
\end{lemma}
\begin{proof}
The Pieri rule (\ref{pieriHL}) implies that the image of $\phi$ is
contained in the linear span of the elements $V_{(m-r,\mu)}$ for
$0\leq r <\frac{m}{2}$ and $\mu$ in $P_r$. Now the linear map $\phi$
is represented by a block triangular matrix with invertible diagonal
matrices as the blocks along the diagonal, and hence is an isomorphism
onto its image.
\end{proof}

Suppose $\mu$ is a partition with $\mu\subset\la$ such that $\la/\mu$
is a horizontal strip.  Let $I$ denote the set of integers $c\geq 1$
such that $\la/\mu$ does (respectively does not) have a box in column 
$c$ (respectively column $c+1$). Define
\[
\varphi_{\la/\mu}(t) = \prod_{c\in I}(1-t^{m_c(\la)})
\]
and observe that for $p\geq \la_1$, the Pieri rule (\ref{pieriHL}) for 
$v_p\cdot V_\la$ may be written in the form
\begin{equation}
\label{pierimirror}
v_p\cdot V_\la  =
\sum_{r\geq 0} \sum_{{\mu\subset\la}\atop{|\mu|=|\la|-r}}
\varphi_{\la/\mu}(t) \, V_{(p+r,\mu)}.
\end{equation}

\begin{thm}
\label{mirrorHL}
For any partition $\la$, we have
\begin{equation}
\label{mirroreqHL}
\sum_{\al\geq 0} (1-t)^{\#\al} \, V_{\la-\al} 
= \sum_{\mu\subset\la}  \varphi_{\la/\mu}(t) \, V_\mu,
\end{equation}
where the first sum is over all compositions $\al$, and the second 
over partitions $\mu\subset\la$ such that $\la/\mu$ is a horizontal
strip.
\end{thm}
\begin{proof}
Choose $p>|\la|$ and let $\ell=\ell(\la)$. Expanding the Giambelli
formula with respect to the first row gives
\[
v_p\cdot V_\la = 
 R_t^{\ell+1}
\,\prod_{j=2}^{\ell+1} \, \frac{1-tR_{1j}}{1-R_{1j}} \, v_{(p,\la)} =
 R_t^{\ell+1}
\,\prod_{j=2}^{\ell+1} (1+(1-t)R_{1j}+(1-t)R_{1j}^2+\cdots)\,v_{(p,\la)}
\]
and therefore
\begin{equation}
\label{pmirror}
v_p\cdot V_\la  =
\sum_{\al\geq 0}(1-t)^{\#\al} \, V_{(p+|\al|,\la-\al)}.
\end{equation}

We compare (\ref{pierimirror}) with (\ref{pmirror}), and 
claim that for every integer $r\geq 0$, 
\begin{equation}
\label{pfeq1HL}
\sum_{|\al|=r} (1-t)^{\#\al} \, V_{(p+r,\la-\al)} 
= \sum_{{\mu\subset\la}\atop{|\mu|=|\la|-r}} 
\varphi_{\la/\mu}(t) \, V_{(p+r,\mu)}.
\end{equation}
The proof is by induction on $r$, with the case $r = 0$ being a tautology.
For the induction step, suppose that we have for some $r > 0$ that
\begin{equation}
\label{pfeq2HL}
\sum_{s\geq r} \sum_{|\al|=s} (1-t)^{\#\al} \, V_{(p+s,\la-\al)} =
\sum_{s\geq r} \sum_{{\mu\subset\la}\atop{|\mu|=|\la|-s}} 
\varphi_{\la/\mu}(t) \, V_{(p+s,\mu)}.
\end{equation}
Expanding the Giambelli formula with respect to the first component,
we obtain $v_{p+s} \, V_{\la-\al}$ as the leading term of
$V_{(p+s,\la-\al)}$, while $v_{p+s} \,V_\mu$ is the leading term of
$V_{(p+s,\mu)}$.  Using these in (\ref{pfeq2HL}) together with Lemma
\ref{prodlemmaHL} proves that
\begin{equation}
\label{pfeq3HL}
\sum_{|\al|=r} (1-t)^{\#\al} \, V_{\la-\al} = 
\sum_{{\mu\subset\la}\atop{|\mu|=|\la|-r}} 
\varphi_{\la/\mu}(t) \, V_\mu.
\end{equation}
Proposition \ref{idHL} now shows that (\ref{pfeq1HL}) is true,
completing the induction. Since we have simultaneously checked that
(\ref{pfeq3HL}) holds for every integer $r\geq 0$, the proof of
Theorem \ref{mirrorHL} is complete.
\end{proof}

\subsection{}
Let $\la$ be any $k$-strict partition of length $\ell$. In this
section, we will obtain the mirror identity to the following version 
of (\ref{toshow}):
\[
\sum_{\al\geq 0} W^{\cC(\la)}_{\la+\al} = 
\sum_{\lambda \to \mu} 2^{N(\lambda,\mu)} \, W^{\cC(\mu)}_\mu 
\]
where the first sum is over all compositions $\al$ of length at most
$\ell+1$, and the second over $k$-strict partitions $\mu$ with
$\la\to\mu$.

\begin{prop}
\label{basicpropC}
Let $\dis \Psi = \prod_{j=2}^{\ell+1} \frac{1-R_{1j}}{1+R_{1j}}$, 
and suppose that we have an equation
\[
\sum_\nu a_\nu w_\nu = \sum_\nu b_\nu w_\nu
\]
in $B^{(k)}$, where the sums are over all $\nu = (\nu_1, \ldots,
\nu_\ell)$, while $a_\nu$ and $b_\nu$ are integers only finitely many
of which are non-zero.  Then we have
\[
\sum_\nu a_\nu \Psi \,w_{(p,\nu)} = \sum_\nu b_\nu \Psi \, w_{(p,\nu)}
\]
in the ring $B^{(k)}$, for any integer $p$.
\end{prop}
\begin{proof}       
It suffices to show that $\sum_\nu c_\nu w_\nu = 0$ implies that
$\sum_\nu c_\nu \Psi \, w_{(p,\nu)} = 0$. We may assume by interchanging rows
that the sum is over partitions $\nu$, because
\[
\Psi \, w_{(p,\nu)} = \Psi \, w_{(p,\tau(\nu))}
\]
for every permutation $\tau$ in the symmetric group $S_\ell$, and
$w_\nu = \Psi \, w_{(p,\nu)} = 0$ whenever $\nu$ has a negative
component. Recall that the $w_\nu$ for $k$-strict partitions $\nu$
form a $\Z$-basis for $B^{(k)}$. Using the relations (\ref{kpresrels})
and induction on the dominance order, we see that for any partition
$\nu$, either $\nu$ is $k$-strict, or $w_\nu$ is a $\Z$-linear
combination of the $w_\mu$ such that $\mu$ is $k$-strict and $\mu
\succ \nu$. 

It follows that we can identify the sum $\sum_\nu c_\nu
w_\nu$ with a $\Z$-linear combination of relations of the form
$\dis\frac{1-R_{h,h+1}}{1+R_{h,h+1}} \, w_\nu$, where $h\geq 1$ and
$\nu$ is a partition such that $\nu_h = \nu_{h+1} > k$. Therefore it
will suffice to show that any such relation
$\dis\frac{1-R_{h,h+1}}{1+R_{h,h+1}} \, w_\nu = 0$ implies that $\dis
\Psi \,\frac{1-R_{h,h+1}}{1+R_{h,h+1}} \, w_{(p,\nu)} = 0$.  For this,
it is enough to check that
\[
\frac{1-R_{1,h}}{1+R_{1,h}}\cdot\frac{1-R_{1,h+1}}{1+R_{1,h+1}}
\cdot\frac{1-R_{h,h+1}}{1+R_{h,h+1}} \, w_{(p,\nu)} = 0.
\] 
But we have
\[
\frac{1-R_{1,h}}{1+R_{1,h}}\cdot\frac{1-R_{1,h+1}}{1+R_{1,h+1}}
\cdot\frac{1-R_{h,h+1}}{1+R_{h,h+1}} =
\frac{1-R_{1,h}}{1+R_{1,h}} - \frac{1-R_{1,h+1}}{1+R_{1,h+1}} +
\frac{1-R_{h,h+1}}{1+R_{h,h+1}}
\]
and since $\nu_h = \nu_{h+1}$, the result is clear.  
\end{proof}

The next result is an easy consequence of the Pieri rule (\ref{finaleq}).

\begin{lemma}[\cite{BKT3}]
\label{stablelm}
Let $\la$ and $\nu$ be $k$-strict partitions such that $\nu_1 >
\max(\la_1,\ell(\la)+2k)$ and $p\geq 0$. Then the coefficient of
$W_\nu$ in the Pieri product $w_p \cdot W_\la$ is equal to the coefficient
of $W_{(\nu_1+1,\nu_2,\nu_3,\ldots)}$ in the product $w_{p+1} \cdot W_\la$.
\end{lemma}

We apply Lemma \ref{stablelm} to make the following important definition.

\begin{defn}
Let $\la$ and $\mu$ be $k$-strict partitions with $\mu\subset\la$,
and choose any $p \gequ \max(\la_1+1,\ell(\la)+2k)$. If
$|\la|=|\mu|+r$ and $\la\to(p+r,\mu)$, then we write
$\mu\rsa\la$ and say that $\la/\mu$
is a {\em $k$-horizontal strip}.  We define $n(\la/\mu):=
N(\la,(p+r,\mu))$; in other words, the numbers $n(\la/\mu)$ are
the exponents that appear in the Pieri product
\begin{equation}
\label{pprod}
w_p\cdot W_\la = \sum_{r,\mu}2^{n(\la/\mu)}\,W_{(p+r,\mu)}
\end{equation}
with the sum over integers $r\geq 0$ and $k$-strict partitions
$\mu\subset\la$ with $|\mu| = |\la|-r$. 
\end{defn}

Note that a $k$-horizontal strip $\la/\mu$ is a pair of partitions
$\la$ and $\mu$ with $\mu\rsa\la$. As such it depends on $\la$ and
$\mu$ and not only on the difference $\la\ssm\mu$. A similar remark
applies to the integer $n(\la/\mu)$ and the polynomials
$\varphi_{\la/\mu}(t)$, $\psi_{\la/\mu}(t)$ of \S \ref{phipsi}.
Observe also that $n(\la/\la) = 0$ and $n(\la/\mu)\geq 1$ whenever
$\la\neq\mu$.  We will study the relation $\mu\rsa\la$ and the integer
$n(\la/\mu)$ in more detail in \S \ref{tableaux}.

\begin{thm}
\label{mirrorC}
For $\la$ any $k$-strict partition we have
\begin{equation}
\label{minuseq}
\sum_{\al\geq 0} 2^{\#\al} \, W^{\cC(\la)}_{\la-\al} 
= \sum_{\mu\rsa\la} 2^{n(\la/\mu)} \, W_\mu,
\end{equation}
where the first sum is over all compositions $\al$.
\end{thm}
\begin{proof}
The argument is essentially the same as that in the proof of Theorem
\ref{mirrorHL}. We choose $p > |\la|+2k$ and prove by induction 
that for each $r\geq 0$,
\[
\sum_{|\al|=r} 2^{\#\al} \, W^{\wh{\cC}}_{(p+r,\la-\al)} 
= \sum_{{\mu\rsa\la}\atop{|\mu|=|\la|-r}} 
2^{n(\la/\mu)}\, W_{(p+r,\mu)},
\]
where $\wh{\cC}=\cC((p,\la))$.  Clearly we have an analogue of Lemma
\ref{prodlemmaHL} which holds for the algebra $B^{(k)}$, and we use
Proposition \ref{basicpropC} as a substitute for Proposition
\ref{idHL}. The point is that for such integers $p$ and $r$, the
raising operator expressions $R^{\wh{\cC}}$ and $R^{(p+r,\mu)}$ both
contain the product $\dis \Psi = \prod_{j=2}^{\ell+1}
\frac{1-R_{1j}}{1+R_{1j}}$, where $\ell=\ell(\la)$.
\end{proof}

\subsection{} 
We now give some consequences of the previous mirror identities for
the Schur $S$- and $Q$-functions. Given any integer sequence $\nu$, 
let $s_\nu(x)$ and $Q_\nu(x)$ denote the Schur functions defined 
in \S \ref{oddsendsHL} and \S \ref{oddsendsCgen}, respectively.
When $t=0$, Theorem \ref{mirrorHL} specializes to the following
well known result.

\begin{cor}
\label{corS1}
For any partition $\la$, we have
\[
\sum_{\al\geq 0} s_{\la-\al}(x) = \sum_\mu s_{\mu}(x),
\]
where the first sum is over all compositions $\al$ and the second
over all partitions $\mu\subset\la$ such that $\la/\mu$ is a 
horizontal strip.
\end{cor} 

The {\em rim} of a partition $\la$ is the set of boxes $[r,c]$ of its
Young diagram such that box $[r+1,c+1]$ lies outside the diagram of
$\la$. If we choose $k > |\la|$ in Theorem \ref{mirrorC}, then we have
$\cC(\la)=\emptyset$ and can thus deduce the following result.

\begin{cor}
\label{corS2}
For any partition $\la$, we have
\[
\sum_{\al\geq 0} 2^{\#\al}\,s_{\la-\al}(x) = \sum_\mu 2^{n(\la/\mu)}
\,s_{\mu}(x),
\]
where the first sum is over all compositions $\al$ and the second
over all partitions $\mu\subset\la$ such that $\la/\mu$ is contained 
in the rim of $\la$, and $n(\la/\mu)$ equals the number of 
edge-connected components of $\la/\mu$.
\end{cor} 

The {\em shifted diagram} of a strict partition $\la$, denoted
$\Sh(\la)$, is obtained from the usual Young diagram by shifting the
$i$-th row $(i-1)$ squares to the right, for every $i>1$. Moreover,
when $\mu$ is a strict partition with $\mu\subset\la$, we set
$\Sh(\la/\mu)=\Sh(\la)\ssm\Sh(\mu)$. We introduce this notation to
emphasize the similarity between Corollary \ref{corS2} and the next
result, which follows by setting $k=0$ in Theorem \ref{mirrorC}.

\begin{cor}
\label{corS3}
For any strict partition $\la$, we have
\[
\sum_{\al\geq 0} 2^{\#\al}\,Q_{\la-\al}(x) = \sum_\mu 
2^{n(\la/\mu)}\,Q_{\mu}(x),
\]
where the first sum is over all compositions $\al$ and the second over
all strict partitions $\mu\subset\la$ such that $\Sh(\la/\mu)$ is
contained in the rim of $\Sh(\la)$, and $n(\la/\mu)$ equals the number
of edge-connected components of $\Sh(\la/\mu)$.
\end{cor}

\subsection{}
\label{recsec}
We next show how the raising operator formalism we have 
developed may be used to obtain a recursion formula for the basis 
elements $W_\la$, expressed in terms of the length of $\la$. 
This question came up naturally during our work on Giambelli formulas 
for the {\em quantum} cohomology
ring of isotropic Grassmannians \cite[\S 1.3]{BKT3}.

Throughout this subsection $\la$ is a partition of length $\ell$ and
$p\geq \la_1$ is an integer. For completeness, we begin with the kind
of recursion we have in mind for the elements $U_{(p,\la)}$ of \S
\ref{typeA}.  We claim that for any partition $\la$ and $p\geq \la_1$,
we have
\begin{equation}
\label{Urec}
U_{(p,\la)} = \sum_{r,\mu}(-1)^r \, u_{p+r} \, U_\mu
\end{equation}
where the sum is over integers $r\geq 0$ and partitions $\mu$ 
obtained from $\la$ by removing a vertical strip with $r$
boxes. Indeed,
Applying the definition (\ref{giambelliA}) gives 
\[
U_{(p,\la)} = \prod_{j>i>1}(1-R_{ij})\prod_{j>1}(1-R_{1j})\, 
u_{(p,\la)} =\sum_{r,\nu}(-1)^r \, u_{p+r} \, U_\nu
\]
where the sum is over integers $r\geq 0$ and compositions 
$\nu$ obtained from $\la$ by removing $r$ boxes, no two 
in the same row. If $\nu$ is not a partition, then we must have
$\nu_{j+1} = \nu_j+1$ for some $j$; hence $U_\nu=0$ by 
Lemma \ref{commuteA}. The result follows.

The analogue of equation (\ref{Urec}) for the $W_{(p,\la)}$ when $p$
and $\la$ are arbitrary appears complicated. The next theorem
gives a top row recursion for sufficiently large $p$. 

\begin{thm}
\label{recurseC}
For any $k$-strict partition $\la$ and $p \geq \max(\la_1+1,\ell(\la)+2k)$,
we have
\begin{equation}
\label{toprecurse}
W_{(p,\la)} = \sum_{r,\mu}(-1)^r \, 2^{n(\la/\mu)} w_{p+r} \,
W_\mu, 
\end{equation} 
where the sum is over $r\geq 0$ and $\mu\rsa\la$
with $|\mu|=|\la|-r$.
\end{thm}
\begin{proof}
The Giambelli formula for $W_{(p,\la)}$ implies that
\[
W_{(p,\la)} =  R^{\ov{\cC}} \, \prod_{j=2}^{\ell+1} 
\frac{1-R_{1j}}{1+R_{1j}} \,w_{(p,\la)}
= \sum_\al (-1)^{|\al|} 2^{\#\al} \, w_{p+|\al|} W^{\cC(\la)}_{\la-\al}
\]
where $\ov{\cC}$ denotes the image of $\cC(\la)$ under the map which
sends $(i,j)$ to $(i+1,j+1)$. We obtain
\[
W_{(p,\la)} = \sum_{r\geq 0} 
(-1)^r w_{p+r} \sum_{|\al|=r} 2^{\#\al} \, W^{\cC(\la)}_{\la-\al}
\]
and then apply the mirror identity (\ref{minuseq}) to finish the proof.
\end{proof}

The reader should compare the stable Pieri rule (\ref{pprod}) with the
recursion formula (\ref{toprecurse}). This is illustrated in the next
example.

\begin{example}
For $k=2$ and $\la=(7,4,2,1)$, we have 
\begin{align*}
w_7\cdot W_{4,2,1} 
&= W_{7,4,2,1} + 2\,W_{8,3,2,1} + 2\,W_{8,4,2} + 2\,W_{9,2,2,1} + 
 2\,W_{9,3,1,1} + 4\,W_{9,3,2} \\
& \quad + 2\,W_{10,2,1,1} + 4\,W_{10,2,2} + 4\,W_{10,3,1}
              + 4\,W_{11,2,1} + 2\,W_{11,3} + 2\,W_{12,2}\,,
\end{align*}
and hence
\begin{align*}
W_{7,4,2,1} 
&= w_7\,W_{4,2,1} - w_8\,(2\,W_{3,2,1} + 2\,W_{4,2}) +
w_9\,(2\,W_{2,2,1} +  2\,W_{3,1,1} + 4\,W_{3,2}) \\
& \quad - w_{10}\,(2\,W_{2,1,1} + 4\,W_{2,2} + 4\,W_{3,1})
      + w_{11}\,(4\,W_{2,1} + 2\,w_3) - 2\,w_{12}\,w_2\,.
\end{align*}
\end{example}

Setting $k=0$ in Theorem
\ref{recurseC} produces the following recursion formula for Schur
$Q$-functions.

\begin{cor}
For any strict partition $\la$ and $p > \la_1$, we have
\begin{equation}
\label{toprecurseQ}
Q_{(p,\la)}(x) = \sum_{r,\mu}(-1)^r \, 2^{n(\la/\mu)} q_{p+r}(x) \, Q_\mu(x)
\end{equation}
where the sum is over $r\geq 0$ and strict partitions $\mu\subset\la$ 
such that $\la/\mu$ is a horizontal $r$-strip, and $n(\la/\mu)$ equals
the number of connected components of $\la/\mu$.
\end{cor}
\noin
Observe that (\ref{toprecurseQ}) is a generalization of Schur's 
identity
\[
Q_{a,b}(x) = q_a(x)\,q_b(x) - 2\,q_{a+1}(x)\,q_{b-1}(x) +
2\,q_{a+2}(x)\,q_{b-2}(x) -\cdots
\]
for any integers $a$, $b$ with $a>b\geq 0$.

It would interesting to prove a top row recursion formula for the
$W_\la$ polynomials analogous to (\ref{toprecurse}) in the general case,
as well as for the $V_\la$ polynomials (in the Hall-Littlewood
theory).  Note that both of these recursions would have to interpolate
between formula (\ref{Urec}), which arises when $k$ is sufficiently
large for the $W_\la$ and when $t=0$ for the $V_\la$, and formula
(\ref{toprecurseQ}), which arises when $k=0$ or $t=-1$, respectively.

\section{Reduction formulas and tableaux}
\label{red&tab}

\subsection{}
\label{tabHL}
The mirror identity (\ref{mirroreqHL}) may be applied in the theory of
Hall-Littlewood functions $Q_\la(x\,;t)$ of \S \ref{oddsendsHL} to
obtain a {\em reduction formula}; what is being reduced is the number
of $x$ variables used in the argument of $Q_\la(x\,;t)$.  Let
$\tilde{x}=(x_2,x_3,\ldots)$ and observe that $q_p(x\,;t) = \sum_{i=0}^p
q_i(x_1\,;t)\,q_{p-i}(\tilde{x}\,;t)$.  Therefore, for any integer
sequence $\la$, we obtain
\begin{equation}
\label{gieq}
q_\la(x\,;t) = \sum_{\al\geq 0} q_{\al}(x_1\,;t)\,q_{\la-\al}(\tilde{x}\,;t)
= \sum_{\al\geq 0} x_1^{|\al|}\,(1-t)^{\#\al}\,q_{\la-\al}(\tilde{x}\,;t)
\end{equation}
summed over all compositions $\al$. If $R$ denotes any raising operator,
we have
\begin{equation*}
R\, q_\la(x\,;t) = q_{R\la}(x\,;t) =
\sum_{\al\geq 0} q_{\al}(x_1\,;t)\,q_{R\la-\al}(\tilde{x}\,;t) 
= \sum_{\al\geq 0} q_{\al}(x_1\,;t)\,R\,q_{\la-\al}(\tilde{x}\,;t).
\end{equation*}
Applying the Giambelli formula to (\ref{gieq}), we thus deduce that
for any partition $\la$, we have
\begin{equation}
\label{midequHL}
Q_\la(x\,;t) = \sum_{\al\geq 0}
x_1^{|\al|}(1-t)^{\#\al}Q_{\la-\al}(\tilde{x}\,;t) = \sum_{p=0}^\infty
x_1^p \sum_{|\al|=p}(1-t)^{\#\al}Q_{\la-\al}(\tilde{x}\,;t)
\end{equation}
and hence, using (\ref{mirroreqHL}), the reduction formula
\begin{equation}
\label{reductionHL}
Q_\la(x\,;t) = \sum_{p=0}^\infty x_1^p 
\sum_{{\mu\subset\la}\atop{|\mu|=|\la|-p}}  \varphi_{\la/\mu}(t) \, 
Q_\mu(\tilde{x}\,;t)
\end{equation}
with the second sum over partitions $\mu\subset\la$ such that $\la/\mu$ is
a horizontal $p$-strip. Repeated application of the reduction equation
(\ref{reductionHL}) results in the {\em tableau formula} of 
\cite[III.(5.11)]{M} for the Hall-Littlewood functions.

%
%

There is an alternative approach to the proof of (\ref{reductionHL})
which applies standard symmetric function theory, as in e.g.\
\cite[III.5]{M}.  If $x'=(x'_1,x'_2,\ldots)$ is a second set of
commuting variables, then the equation
\[
Q_\la(x,x'\,;t) = \sum_\mu Q_{\la/\mu}(x\,;t) \, Q_{\mu}(x'\,;t)
\]
summed over partitions $\mu\subset\la$ may be used to define the 
{\em skew Hall-Littlewood functions} $Q_{\la/\mu}(x\,;t)$. In particular,
this gives
\[
Q_\la(x_1,\tilde{x}\,;t) = \sum_\mu
Q_{\la/\mu}(x_1\,;t)\,Q_\mu(\tilde{x}\,;t). 
\]
According to \cite[III.(5.14)]{M}, we have
\begin{equation}
\label{maceq}
Q_{\la/\mu}(x_1\,;t) = \varphi_{\la/\mu}(t) \, x_1^{|\la-\mu|} 
\end{equation}
and hence (\ref{reductionHL}) is established. By comparing with
(\ref{midequHL}), we can apply this reasoning to prove the mirror
identity (\ref{mirroreqHL}) for the Hall-Littlewood functions.  Note
however that the proof of (\ref{maceq}) uses the inner product
of \cite[III.4]{M}, and the point here is that this is not needed in
order to establish Theorem \ref{mirrorHL} (which of course does not
involve the variables $x$). The raising operator approach will be
exploited further in the next subsections when we study theta
polynomials.

\subsection{}
\label{theta1}
Let $y=(y_1,\ldots,y_k)$ and consider the theta polynomials
$\ti_r(x\,;y)$ and $\Ti_\la(x\,;y)$ defined in \S \ref{oddsendsCgen},
so that $\Ti_\la = R^\la\, \ti_\la$ for any $k$-strict partition
$\la$. We compute that
\[
\sum_{r=0}^{\infty}\ti_r(x\,;y)t^r = \frac{1+x_1t}{1-x_1t}\,
\prod_{i=2}^{\infty}\frac{1+x_it}{1-x_it} \,\prod_{j=1}^k
(1+y_jt) = \sum_{i=0}^{\infty}2^{\#i}\,x_1^i\,\sum_{r=0}^{\infty}
\ti_r(\tilde{x}\,;y)
\]
and therefore, for any $k$-strict partition $\la$, 
\begin{equation}
\label{stepto}
\ti_\la(x\,;y) = \sum_{\al\geq0}
x_1^{|\al|}\,2^{\#\al}\,\ti_{\la-\al}(\tilde{x}\,;y)
\end{equation}
summed over all compositions $\al$. Applying the raising operator 
$R^\la$ to both sides of (\ref{stepto}) produces 
\[
\Ti_\la(x\,;y) = \sum_{\al\geq 0} x_1^{|\al|}\,2^{\#\al}\,
\Ti^{\cC(\la)}_{\la-\al}(\tilde{x}\,;y) = 
\sum_{p=0}^\infty \,x_1^p \,\sum_{|\al|=p}2^{\#\al}\,
\Ti^{\cC(\la)}_{\la-\al}(\tilde{x}\,;y),
\]
where $\Ti^{\cC(\la)}_{\la-\al} = R^\la\,\ti_{\la-\al}$ by
definition. We now use the mirror identity (\ref{minuseq}) to deduce
the next result.
\begin{thm}
\label{reductthm}
For any $k$-strict partition $\la$, we have the reduction formula 
\begin{equation}
\label{reductTH}
\Ti_\la(x\,;y) = \sum_{p=0}^\infty \,x_1^p\,
\sum_{{\mu\rsa\la}\atop{|\mu|=|\la|-p}}
2^{n(\la/\mu)} \, \Ti_\mu(\tilde{x}\,;y).
\end{equation}
\end{thm}

\subsection{}
\label{tableaux}
In this subsection we will apply the reduction formula
(\ref{reductTH}) to obtain a tableau description of the theta
polynomials $\Ti_\la$, where the tableaux in question are fillings of
the Young diagram of $\la$.  
We say that boxes $[r,c]$ and $[r',c']$ are {\em $k'$-related} if
$|c-k-\frac{1}{2}|+r = |c'-k-\frac{1}{2}|+r'$ (we think of $k'$ as
being equal to $k-\frac{1}{2})$. For example, the two grey boxes in 
the diagram of Figure \ref{kprrelated} are $k'$-related.
\begin{figure}
\centering
\pic{0.65}{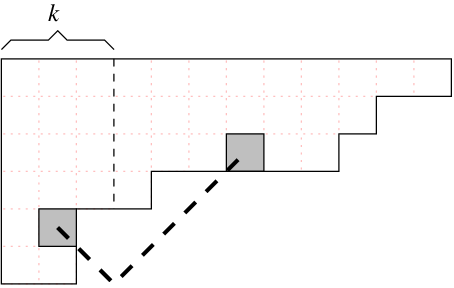} 
\caption{Two $k'$-related boxes in a Young diagram}
\label{kprrelated}
\end{figure}
We call box $[r,c]$ a {\em left box} if $c \leq k$ and a {\em
right box} if $c>k$.

If $\mu\subset\la$ are two $k$-strict
partitions such that $\la/\mu$ is a $k$-horizontal strip,
we define $\la_0=\mu_0=\infty$ and agree that the
diagrams of $\la$ and $\mu$ include all boxes $[0,c]$ in row zero.
%
We let $R$ (respectively $\A$) denote the set of
right boxes of $\mu$ (including boxes in row zero) which are bottom
boxes of $\la$ in their column and are (respectively are not)
$k'$-related to a left box of $\la/\mu$. 

\begin{lemma}
A pair $\mu\subset\la$ of $k$-strict partitions forms a $k$-horizontal
strip $\la/\mu$ if and only if (i) $\la/\mu$ is contained in the rim
of $\la$, and the right boxes of $\la/\mu$ form a horizontal strip;
(ii) no two boxes in $R$ are $k'$-related; and (iii) if two boxes of
$\la/\mu$ lie in the same column, then they are $k'$-related to
exactly two boxes of $R$, which both lie in the same row. The integer
$n(\la/\mu)$ is equal to the number of connected components of $\A$
which do not have a box in column $k+1$.
\end{lemma}
\begin{proof}
We have $\mu\rsa\la$ if and only if $\la\to(p+r,\mu)$ for
any $p>|\la|+2k$, where $r=|\la-\mu|$. Observe that a box of
$(p+r,\mu)\ssm\la$ corresponds to a box of $\mu$ which is a bottom box
of $\la$ in its column. The fact that $\mu\rsa\la$ is
characterized by conditions (i)--(iii) and that
$n(\la/\mu)=N(\la,(p+r,\mu))$ is computed as claimed is now an easy
translation of the definitions in \S \ref{pierirulegen}.
\end{proof}

Let $\la$ and $\mu$ be any two $k$-strict partitions with $\mu\subset\la$.

\begin{defn}
a) A {\em $k$-tableau} $T$ of shape $\la/\mu$ is a sequence of $k$-strict
partitions
\[
\mu = \la^0\subset\la^1\subset\cdots\subset\la^r=\la
\]
such that $\la^i/\la^{i-1}$ is a $k$-horizontal strip for $1\leq i\leq
r$.  We represent $T$ by a filling of the boxes in $\la/\mu$ with
positive integers which is weakly increasing along each row and down
each column, such that for each $i$, the boxes in $T$ with entry $i$
form the skew diagram $\la^i/\la^{i-1}$. A {\em standard $k$-tableau}
on $\la/\mu$ is a $k$-tableau $T$ of shape $\la/\mu$ such that the
entries $1,2,\ldots,|\la - \mu|$ each appear once in $T$.  For any
$k$-tableau $T$ we define
\[
n(T)=\sum_i n(\la^i/\la^{i-1}) \quad \text{and} \quad
x^T = \prod_i x_i^{m_i}
\] 
where $m_i$ denotes the number of times that $i$ appears in $T$. 

\medskip
\noin b) Let {\bf P} denote the ordered alphabet
$\{1'<2'<\cdots<k'<1<2<\cdots\}$.  The symbols $1',\ldots,k'$ are said
to be {\em marked}, while the rest are {\em unmarked}.  A {\em
$k$-bitableau} $U$ of shape $\la$ is a filling of the boxes in $\la$
with elements of {\bf P} which is weakly increasing along each row and
down each column, such that (i) the marked entries are strictly
increasing along each row, and (ii) the unmarked entries form a
$k$-tableau $T$. We define
\[
n(U)=n(T) \quad \text{and} \quad
(xy)^U= x^T \,\prod_{j=1}^k y_j^{n_j} 
\]
where $n_j$ denotes the number of times that $j'$ appears in $U$.
\end{defn}

\begin{thm}
\label{tableauxthm}
For any $k$-strict partition $\la$, we have 
\begin{equation}
\label{tableauxeq}
\Ti_\la(x\,;y) = \sum_U 2^{n(U)}(xy)^U 
\end{equation}
where the sum is over all $k$-bitableaux $U$ of shape $\la$ 
\end{thm}
\begin{proof}
Let $m$ be a positive integer, $x^{(m)} = (x_1,\dots,x_m)$, and
let  $\Ti_\la(x^{(m)}\,;y)$ be the result of substituting
$x_i=0$ for $i>m$ in $\Ti_\la(x\,;y)$. It follows from equation
(\ref{reductTH}) that
\begin{equation}
\label{mreduct}
\Ti_\la(x^{(m)}\,;y) = \sum_{p=0}^\infty \,x_m^p\,
\sum_{{\mu\rsa\la}\atop{|\mu|=|\la|-p}}
2^{n(\la/\mu)} \, \Ti_\mu(x^{(m-1)}\,;y).
\end{equation}
Iterating equation (\ref{mreduct}) $m$ times produces
\[
\Ti_\la(x^{(m)}\,;y) = \sum_{\mu,T} \,2^{n(T)}\,x^T
\Ti_\mu(0\,;y)
\]
where the sum is over all $k$-strict partitions $\mu\subset\la$ and 
$k$-tableau $T$ of shape $\la/\mu$ with no entries greater than $m$, 
and $\Ti_\mu(0\,;y)$ is obtained from $\Ti_\mu(x\,;y)$ by substituting $x_i=0$
for all $i$. The raising operator definition of $\Ti_\mu$ gives
\begin{equation}
\label{Ti0mu}
\Ti_\mu(0\,;y) = R^\mu\, e_\mu(y)
\end{equation}
where $e_\mu = \prod_ie_{\mu_i}(y)$ and $e_r(y)$ denotes the $r$-th 
elementary symmetric polynomial in $y$. Since $e_r(y)=0$ for $r>k$, 
we deduce from (\ref{Ti0mu}) that $\Ti_\mu(0\,;y)=0$ unless 
$\mu$ is contained in the first $k$ columns. But in this case we 
have $\cC(\mu)=\emptyset$ and 
\[
\Ti_\mu(0\,;y) = \prod_{i<j}(1-R_{ij})\, e_\mu(y) = 
s_{\mu'}(y)
\]
where $\mu'$ is the partition conjugate to $\mu$. The combinatorial 
definition of Schur $S$-functions \cite[I.(5.12)]{M} states that
\begin{equation}
\label{seq}
s_{\mu'}(y) = \sum_S y^S
\end{equation}
summed over all semistandard Young tableaux $S$ of shape $\mu'$ with
entries from $1$ to $k$. In the spirit of this article, the identity
(\ref{seq}) may be derived using raising operators, by iterating the
specialization of equation (\ref{reductionHL}) at $t=0$. We conclude
that
\[
\Ti_\la(x^{(m)}\,;y) = \sum_U \,2^{n(U)}\,(xy)^U
\]
summed over all $k$-bitableaux $U$ of shape $\la$ with no entries greater
than $m$. The result follows by letting $m\to\infty$.
\end{proof}

\begin{example} Let $k=1$, $\la=(3,1)$, and consider the alphabet
$\text{\bf P}_{1,2}=\{1'<1<2\}$. There are twelve $k$-bitableaux $T$
of shape $\la$ with entries in $\text{\bf P}_{1,2}$. The bitableau
$T=\dis \begin{array}{l} 1'\,1\,2 \\ 1 \end{array}$ satisfies
$n(T)=3$, the seven bitableaux
\[
\begin{array}{l} 1\,1\,2 \\ 1 \end{array}  \ \
\begin{array}{l} 1\,1\,2 \\ 2 \end{array}  \ \
\begin{array}{l} 1\,2\,2 \\ 1 \end{array}  \ \
\begin{array}{l} 1\,2\,2 \\ 2 \end{array}  \ \
\begin{array}{l} 1'\,1\,2 \\ 2 \end{array}  \ \
\begin{array}{l} 1'\,2\,2 \\ 1 \end{array}  \ \
\begin{array}{l} 1'\,1\,2 \\ 1' \end{array}  
\]
satisfy $n(T)=2$, while the four bitableaux
\[
\begin{array}{l} 1'\,1\,1 \\ 1 \end{array}  \ \
\begin{array}{l} 1'\,2\,2 \\ 2 \end{array}  \ \
\begin{array}{l} 1'\,1\,1 \\ 1' \end{array} \ \
\begin{array}{l} 1'\,2\,2 \\ 1' \end{array} 
\]
satisfy $n(T)=1$. We deduce from Theorem \ref{tableauxthm} that
\begin{align*}
\Ti_{3,1}(x_1,x_2\,; y_1) 
&= (4x_1^3x_2+8x_1^2x_2^2+4x_1x_2^3) 
+ (2x_1^3+8x_1^2x_2+8x_1x_2^2+2x_2^3)\,y_1 \\
&\qquad + (2x_1^2+4x_1x_2+2x_2^2)\,y_1^2 \\
&= Q_{3,1}(x_1,x_2) + 
\left(Q_3(x_1,x_2) + Q_{2,1}(x_1,x_2)\right)\,y_1 + 
Q_2(x_1,x_2)\, y_1^2.
\end{align*}
\end{example}

\subsection{}
\label{Ftableaux}
Notice that (\ref{tableauxeq}) may be rewritten as
\[
\Ti_\la(x\,;y) = \sum_{\mu,T}2^{n(T)}\,x^T\,s_{\mu'}(y)
\]
with the sum over all partitions $\mu\subset\la$ and
$k$-tableau $T$ of shape $\la/\mu$. This motivates the 
following definition.

\begin{defn}
\label{Fdefn}
For $\la$ and $\mu$ any two $k$-strict partitions with $\mu\subset\la$,
let
\[
F_{\la/\mu}^{(k)}(x)=\sum_T 2^{n(T)}\,x^T
\]
where the sum is over all $k$-tableaux $T$ of shape $\la/\mu$.
\end{defn}

\begin{cor}
Let $x=(x_1,x_2,\ldots)$ and $x'=(x'_1,x'_2,\ldots)$ be two sets of 
variables, and let $\la$ be any $k$-strict partition. Then we have
\begin{equation}
\label{mastereq}
\Ti_\la(x,x'\,;y) = \sum_{\mu\subset\la} F^{(k)}_{\la/\mu}(x)\,
\Ti_\mu(x'\,;y)\, ,
\end{equation}
\begin{equation}
\label{master2}
\Ti_\la(x\,;y) = \sum_{\mu\subset\la} F^{(k)}_{\la/\mu}(x)\,
s_{\mu'}(y)\, ,
\end{equation}
and 
\begin{equation}
\label{master3}
F^{(k)}_\la(x,x') = \sum_{\mu\subset\la} F^{(k)}_{\la/\mu}(x)\,
F^{(k)}_{\mu}(x'),
\end{equation}
where the sums are over all $k$-strict partitions $\mu\subset\la$.
\end{cor}
\begin{proof}
Let $m$ be a positive integer, $x^{(m)} = (x_1,\dots,x_m)$, and let
$F^{(k)}_{\la/\mu}(x^{(m)})$ (respectively $\Ti_\la(x^{(m)}\,;y)$) be
the result of substituting $x_i=0$ for $i>m$ in $F^{(k)}_{\la/\mu}(x)$
(respectively $\Ti_\la(x\,;y)$). Applying equation (\ref{mreduct}) 
as in the proof of Theorem \ref{tableauxthm} gives
\[
\Ti_\la(x^{(m)},x'\,;y) = \sum_{\mu\subset\la} \,F_{\la/\mu}^{(k)}(x^{(m)})\,
\Ti_\mu(x'\,;y).
\]
Now let $m\to\infty$ to obtain (\ref{mastereq}). Equations (\ref{master2})
and (\ref{master3}) are deduced from (\ref{mastereq}) by substituting 
$x'=0$ and $y=0$, respectively.
\end{proof}

The polynomials $\Ti_\mu(x'\,;y)$ for $\mu$ a $k$-strict partition 
form a free $\Z$-basis for the ring 
$\Gamma^{(k)}(x'\,; y)$ of theta polynomials in $x'$ and $y$. The 
identity (\ref{mastereq}) states that $F^{(k)}_{\la/\mu}(x)$ is 
the coefficient of $\Ti_\mu(x'\,;y)$ when $\Ti_\la(x,x'\,;y)$ is 
expanded in this basis. We deduce that $F^{(k)}_{\la/\mu}(x)$ is 
a {\em symmetric function} in the variables $x$. The identity
(\ref{master3}) may be further generalized as follows.

\begin{cor}
For any two $k$-strict partitions $\la$ and $\mu$ with $\mu\subset\la$,
we have
\begin{equation}
\label{FFF}
F^{(k)}_{\la/\mu}(x,x') = \sum_{\nu} F^{(k)}_{\la/\nu}(x)\,
F^{(k)}_{\nu/\mu}(x'),
\end{equation}
where the sum is over all $k$-strict partitions $\nu$
with $\mu\subset\nu\subset\la$.
\end{cor}
\begin{proof}
Let $z=(z_1,z_2,\ldots)$ be a third set of variables. We have
\begin{gather*}
\sum_{\mu\subset\la} F^{(k)}_{\la/\mu}(x,x')\, \Ti_{\mu}(z\,;y) =
\Ti_\la(x,x',z\,;y) = 
\sum_\nu F^{(k)}_{\la/\nu}(x)\,\Ti_\nu(x',z\,;y) \\
= \sum_{\mu,\nu}
F^{(k)}_{\la/\nu}(x)\,F^{(k)}_{\nu/\mu}(x')\,\Ti_{\mu}(z\,;y).
\end{gather*}
We now equate the coefficients of $\Ti_{\mu}(z\,;y)$ at either end
of the above equalities.
\end{proof}

In the next section, we will show that the $k$-strict partitions
$\nu$ which contribute positively to the sum in equation (\ref{FFF})
correspond to the reduced factorizations of a certain {\em skew
element} in the hyperoctahedral group.

\begin{example}
\label{4ex}
a) Suppose that $k=0$ or, more generally, that $\mu_i \geq
\min(k,\la_i)$ for all $i$. Then Definition \ref{Fdefn} becomes the
tableau based definition of skew Schur $Q$-functions
\cite[III.(8.16)]{M}, hence $F^{(k)}_{\la/\mu}(x)=Q_{\la/\mu}(x)$.

\medskip
\noin
b) Suppose that $\cC(\la)=\emptyset$, so in particular $\la_i>k$ 
implies that $i=1$. One then checks that $n(\la/\mu)$
is equal to the number of edge-connected components of $\la/\mu$. 
For any partition $\mu\subset\la$, Worley \cite[\S 2.7]{W} has shown 
that 
\[
\sum_T 2^{n(T)}\, x^T = S_{\la/\mu}(x),
\]
where the symmetric function $S_{\la/\mu}(x)$ satisfies
\[
S_{\la/\mu}(x) = \det(q_{\la_i-\mu_j+j-i}(x))_{i,j}.
\]
We therefore have $F^{(k)}_{\la/\mu}(x) =  S_{\la/\mu}(x)$ and
\[
\Ti_\la(x\,;y) = \sum_{\mu\subset\la}S_{\la/\mu}(x)\,s_{\mu'}(y),
\]
in agreement with \cite[\S 5.5]{BKT2}.
\end{example}

\medskip
\noin c) It is clear from Definition \ref{Fdefn} that in general, the
skew function $F^{(k)}_{\la/\mu}(x)$ depends on both $\la$ and $\mu$,
and not only on the difference $\la\ssm\mu$. For instance, we have
$F^{(1)}_{(3,2)/(3)}(x) = 0$, while $F^{(1)}_{(r,2)/(r)}(x)=Q_2(x)$
for any $r>3$. On the other hand,
$F^{(1)}_{(5,4,1,1)/(4,3)}(x)=F^{(1)}_{(\la,5,4,1,1)/(\la,4,3)}(x)=0$
for any strict partition $\la$ with $\la_\ell>5$. Criteria for the
vanishing of $F^{(k)}_{\la/\mu}(x)$ are given in Proposition
\ref{abprop} and Corollary \ref{bhcor}.

\medskip
\noin 
d) Suppose that there is only one variable $x$. Then we have
\[
F^{(k)}_{\la/\mu}(x) = \begin{cases}
2^{n(\la/\mu)}\, x^{|\la-\mu|} & \text{if $\la/\mu$ is a 
$k$-horizontal strip}, \\
0 & \text{otherwise}.
\end{cases}
\]
The reader should compare this to (\ref{maceq}).

\subsection{}
\label{theta4}

Observe that we have 
\[
\ti_p(x,x';y) = \sum_{i=0}^p q_i(x)\,\ti_{p-i}(x';y)
\]
and therefore, for any $k$-strict partition $\la$,
\[
\ti_\la(x,x';y) = \sum_{\al\geq 0} q_\al(x)\,\ti_{\la-\al}(x';y) 
\]
summed over all compositions $\al\geq 0$. It follows that
\begin{equation}
\label{compeq}
\Ti_\la(x,x';y) = \sum_{\al\geq 0} 
q_\al(x)\,\Ti^{\cC(\la)}_{\la-\al}(x';y).
\end{equation}
For each fixed composition $\al$, the product
$q_\al(x)=\prod_iq_{\al_i}(x)$ is a nonnegative linear combination of
Schur $Q$-functions. In addition, $\Ti^{\cC(\la)}_{\la-\al}(x\,;y)$ is
a $\Z$-linear combination of the $\Ti_\mu(x'\,;y)$ for $k$-strict
partitions $\mu$. By comparing (\ref{compeq}) with (\ref{mastereq}),
we deduce that for every two $k$-strict partitions $\la$, $\mu$ with
$\mu\subset\la$, the function $F^{(k)}_{\la/\mu}(x)$ is a linear
combination of Schur $Q$-functions with integer coefficients. In the
next section, we will prove that these coefficients are {\em
nonnegative} integers.

\section{Stanley symmetric functions and skew elements}
\label{sss}

\subsection{}
\label{ssfns}
Let $B_n$ be the hyperoctahedral group of signed permutations on the
set $\{1,\ldots,n\}$, and $B_{\infty} = \cup_nB_n$. We will adopt the
notation where a bar is written over an entry with a negative sign.
The group $B_{\infty}$ is generated by the simple transpositions
$s_i=(i,i+1)$ for $i>0$, and the sign change $s_0(1)=\ov{1}$. Every
element $w\in B_\infty$ can be expressed as a product of simple
reflections $s_i$; any such expression of minimal length is called a
{\em reduced word} for $w$. The length of $w$, denoted $\ell(w)$, is
the length of any reduced word for $w$. A factorization $w=uv$ in
$B_\infty$ is {\em reduced} if $\ell(w)=\ell(u)+\ell(v)$.

Following \cite{FS} and \cite{FK1, FK2}, we will use the nilCoxeter
algebra $\cB_n$ of the hyperoctahedral group $B_n$ to define type C
Stanley symmetric functions and Schubert polynomials. $\cB_n$ is the
free associative algebra with unity generated by the elements
$u_0,u_1,\ldots,u_{n-1}$ modulo the relations
\[
\begin{array}{rclr}
u_i^2 & = & 0 & i\geq 0\ ; \\
u_iu_j & = & u_ju_i & |i-j|\geq 2\ ; \\
u_iu_{i+1}u_i & = & u_{i+1}u_iu_{i+1} & i>0\ ; \\
u_0u_1u_0u_1 & = & u_1u_0u_1u_0.
\end{array}
\]
For any $w\in B_n$, choose a reduced word $s_{a_1}\cdots s_{a_\ell}$
for $w$ and define $u_w = u_{a_1}\ldots u_{a_\ell}$. Since the last
three relations listed are the Coxeter relations for $B_n$, it is
clear that $u_w$ is well defined, and that the $u_w$ for $w\in B_n$
form a free $\Z$-basis of $\cB_n$.

Let $\omega$ be an indeterminate and define
\[
A_i(\omega) = (1+\omega u_{n-1})(1+\omega u_{n-2})\cdots 
(1+\omega u_i) \ ;
\]
\[
C(\omega) = (1+\omega u_{n-1})\cdots(1+\omega u_1)
(1+2\omega u_0)(1+\omega u_1)\cdots (1+\omega u_{n-1}).
\]
Set $x=(x_1,x_2,\ldots)$ and consider the formal product
$C(x):=C(x_1)C(x_2)\cdots$.  Arguing as in \cite[Prop.\ 4.2]{FK2}, we
see that the relation $C(x_i)C(x_j) = C(x_j)C(x_i)$ holds for all
indices $i$ and $j$.  We deduce that the functions $F_w(x)$ in the
formal power series expansion
\begin{equation}
\label{defF}
C(x) = \sum_{w\in B_n}F_w(x)\, u_w
\end{equation}
are symmetric functions in $x$. The $F_w$ are the {\em type C Stanley 
symmetric functions}, introduced and studied in \cite{BH, FK2, L}.

Let $y=(y_1,y_2,\ldots)$. The Billey-Haiman {\em type C Schubert
polynomials} $\CS_w(x\,;y)$ for $w\in B_n$ are defined by expanding
the formal product
\begin{equation}
\label{defC}
C(x) A_1(y_1)A_2(y_2)\cdots A_{n-1}(y_{n-1})
= \sum_{w\in B_n}\CS_w(x\,;y)\, u_w.
\end{equation}
The above definition is equivalent to the one in \cite{BH}, as is
shown in \cite[\S 7]{FK2}. One checks that $\CS_w$ is stable 
under the natural inclusion of $B_n$ in $B_{n+1}$, and hence 
well defined for $w\in B_\infty$. 
We also deduce the following result from (\ref{defF}) and (\ref{defC}).

\begin{prop}
Let $w\in B_\infty$ and $x'=(x'_1,x'_2,\ldots)$. Then we have
\begin{equation}
\label{Fxx}
F_w(x,x') = \sum_{uv=w} F_u(x)F_v(x')
\end{equation}
and 
\begin{equation}
\label{besteq}
\CS_w(x,x'\,;y) = \sum_{uv=w}F_u(x)\,\CS_v(x'\,;y)
\end{equation} 
where the sums are over all reduced factorizations $uv=w$ in $B_\infty$.
\end{prop}

\subsection{}
\label{wsec}
We say that an element $w\in B_\infty$ is {\em $k$-Grassmannian} if it
is Grassmannian with respect to the simple reflection $s_k$, i.e., if
$\ell(ws_i)=\ell(w)+1$ for all $i\neq k$. Given any $k$-strict
partition $\la$, we will define a $k$-Grassmannian element $w_\la\in
B_\infty$. If $\cP(k,n)$ denotes the set of $k$-strict
partitions whose Young diagrams fit inside an $(n-k)\times (n+k)$
rectangle, then $w_\la\in B_n$ for any $n$ such that $\la\in
\cP(k,n)$.

The signed permutation $w_{\la}=(w_1,\ldots,w_n)$ has a unique descent
at $k$, that is, $w(i) < w(i+1)$ whenever $i\neq k$ and the first
$k$ entries of $w_\la$ are positive. For $\lambda \in \cP(k,n)$ we let
$\la^1$ be the strict partition formed by the boxes of $\lambda$ in
columns $k+1$ through $k+n$.  The negative entries of $w_{\la}$ are
then given by the parts of $\la^1$.  Consider the shape $\Pi(k,n)$
obtained by attaching an $(n-k) \times k$ rectangle to the left side
of a staircase partition with $n$ rows.  When $n=7$ and $k=3$, this
looks as follows.
\[ \Pi(k,n) \ \ = \ \ \ \raisebox{-36pt}{\pic{.6}{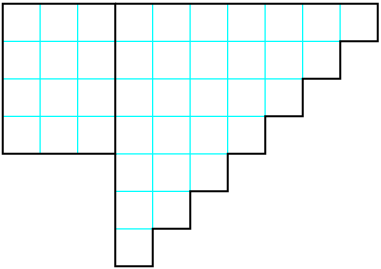}} \]
The diagram of $\la$ can be placed inside $\Pi(k,n)$ so that the
northwest corners of $\la$ and $\Pi(k,n)$ coincide.  The boxes of the
staircase partition which are outside $\lambda$ then fall into
south-west to north-east diagonals.  The first $k$ (respectively, the
last $n-k-\ell(\la^1)$) entries of $w_\la$ are the lengths of the
diagonals which are (respectively, are not) $k$-related to one of the
bottom boxes in the first $k$ columns of $\lambda$.  For example, the
partition $\lambda = (8,5,2,1)\in \cP(3,7)$ results in the element
$w_\la = 147\ov{5}\ov{2}36$.
\[
\lambda \ = \ \ \raisebox{-53pt}{\pic{.6}{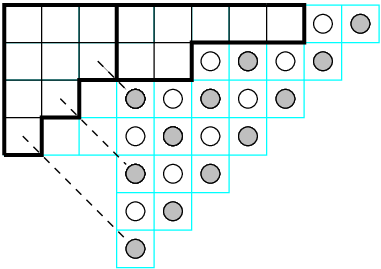}}
\]
Observe that $w_\la$ is stable under the inclusion of $W_n$ into
$W_{n+1}$, and thus is well defined as an element of $W_\infty$.

Conversely, for any $k$-Grassmannian element $w\in B_n$ there 
exist strict partitions $u$, $\zeta$, $v$ of lengths $k$, $r$, 
and $n-k-r$, respectively, so that 
\[
w=(u_k,\ldots,u_1,
\ov{\zeta}_1,\ldots,\ov{\zeta}_r,v_{n-k-r},\ldots,v_1).
\]
Define $\al_i$ for $1\leq i \leq k$ by
\[
\al_i=u_i+i-k-1+\#\{j\ |\ \zeta_j > u_i\}.
\]
Then $w=w_\la$ for the partition $\la\in\cP(k,n)$ with $\la^1=\zeta$
and such that the lengths of the first $k$ columns of $\la$ are given
by $\al_1,\ldots,\al_k$.

It was proved in \cite[\S 6]{BKT2} that
\begin{equation}
\label{TSeq}
\Ti_\la(x\,;y) = \CS_{w_\la}(x\,;y)
\end{equation}
for any $k$-strict partition $\la$.

\subsection{} 
The following definition can be formulated for any Coxeter group, 
but the name is justified by its application in the case of the finite
classical Weyl groups.

\begin{defn}
An element $w\in B_\infty$ is called {\em skew} if 
there exists a $k$-strict partition $\la$ (for some $k$)
and a reduced factorization $w_\la = ww'$ in $B_\infty$.
\end{defn}

Note that if we have a reduced factorization $w_\la = ww'$ in $B_n$
for some partition $\la\in \cP(k,n)$, then the right factor $w'$ is
$k$-Grassmannian, and therefore equal to $w_\mu$ for some partition
$\mu\in \cP(k,n)$.

\begin{prop}
\label{skewprop}
Suppose that $w$ is a skew element of $B_\infty$, and let $\la$ and $\mu$ be 
$k$-strict partitions such that the factorization $w_\la = ww_\mu$
is reduced. Then we have
$\mu\subset\la$ and $F_w(x) = F^{(k)}_{\la/\mu}(x)$. 
\end{prop}
\begin{proof}
By combining (\ref{TSeq}) with (\ref{mastereq}) and (\ref{besteq}),
we see that 
\begin{equation}
\label{3eq}
\sum_{\mu\subset\la} F^{(k)}_{\la/\mu}(x)\, \Ti_\mu(x'\,;y) =
\Ti_\la(x,x'\,; y) = 
\sum_{uv=w_\la}F_u(x)\,\CS_v(x'\,;y)
\end{equation}
where the second sum is over all reduced factorizations $uv=w_\la$. In
any such factorization, the right factor $v$ is equal to $w_\nu$ for
some $k$-strict partition $\nu$, and therefore $\CS_v(x'\,;y) =
\Ti_\nu(x'\,;y)$.  Since the $\Ti_\nu(x'\,;y)$ for $\nu$ a $k$-strict
partition form a $\Z$-basis for the ring $\Gamma^{(k)}(x'\,; y)$ of
theta polynomials in $x'$ and $y$, the desired result follows
immediately.
\end{proof}


The Pieri rule (\ref{finaleq}) illustrates that the order relation on
$k$-strict partitions given by the inclusion of Young diagrams is not
compatible with the Bruhat order on $B_\infty$. However, Proposition
\ref{skewprop} shows that the {\em weak} Bruhat order on the
$k$-Grassmannian elements of $B_\infty$ respects the inclusion 
of $k$-strict diagrams.

\begin{defn}
Let $\la$ and $\mu$ be $k$-strict partitions in $\cP(k,n)$ with
$\mu\subset\la$.  We say that $(\la,\mu)$ is a {\em compatible pair}
if there is a reduced word for $w_\la$ whose last $|\mu|$ entries form
a reduced word for $w_\mu$; equivalently, if we have $\ell(w_\la
w_\mu^{-1}) = |\la - \mu|$. In other words, $(\la,\mu)$ is a
compatible pair if $w_\la$ exceeds $w_\mu$ in the weak Bruhat order.
\end{defn}

\begin{cor}
\label{compcor}
Let $(\la,\mu)$ be a compatible pair. 
Then there is a 1-1 correspondence between reduced factorizations
of $w_\la w_\mu^{-1}$ and $k$-strict partitions $\nu$ with 
$\mu\subset\nu\subset\la$ such that $(\la,\nu)$ and $(\nu,\mu)$ are
compatible pairs.
\end{cor}
\begin{proof}
Suppose that $w_\la = ww_\mu$ and let $w=w'w''$ be a reduced 
factorization. Then $w''w_\mu$ is $k$-Grassmannian and hence 
equal to $w_\nu$ for a unique $k$-strict partition $\nu$.
Proposition \ref{skewprop} implies that $\mu\subset\nu\subset\la$, 
and clearly the pairs $(\la,\nu)$ and $(\nu,\mu)$ are compatible.
The converse is obvious.
\end{proof}

\begin{thm}
\label{abprop}
Let $\la$ and $\mu$ be $k$-strict partitions in $\cP(k,n)$ with
$\mu\subset\la$.  Then the following conditions are equivalent: {\em
(a)} $F^{(k)}_{\la/\mu}(x)\neq 0$; {\em (b)} $(\la,\mu)$ is a
compatible pair; {\em (c)} there exists a standard $k$-tableau on
$\la/\mu$. If any of these conditions holds, then
$F^{(k)}_{\la/\mu}(x) = F_{w_\la w_\mu^{-1}}(x)$.
%
%
\end{thm}
\begin{proof}
Equation (\ref{3eq}) may be rewritten in the form
\begin{equation}
\label{2eq}
\sum_{\mu\subset\la} F^{(k)}_{\la/\mu}(x)\, \Ti_\mu(x'\,;y) =
\sum_{\mu} F_{w_\la w_\mu^{-1}}(x)\,\Ti_\mu(x'\,;y)
\end{equation}
where the second sum is over all $\mu\subset\la$ such that $(\la,\mu)$
is a compatible pair. It follows that
\[
F^{(k)}_{\la/\mu}(x) =
\begin{cases}
F_{w_\la w_\mu^{-1}}(x) & \text{if $(\la,\mu)$ is a compatible pair}, \\
0 & \text{otherwise}
\end{cases}
\]
and hence that (a) and (b) are equivalent. Suppose now
that $(\la,\mu)$ is a compatible pair with $|\la|=|\mu|+1$, so that
$w_\la = s_iw_\mu$ for some $i\geq 0$. Observe that if $x$ is a single
variable, then $F_{s_i}(x) = 2x$, and therefore $F^{(k)}_{\la/\mu}(x)=
2x\neq 0$. We deduce from Example \ref{4ex}(d) that $\la/\mu$ must be
a $k$-horizontal strip.  Using Corollary \ref{compcor}, it follows
that there is a 1-1 correspondence between reduced words for $w_\la
w_\mu^{-1}$ and sequences of $k$-strict partitions
\[
\mu = \la^0\subset\la^1\subset\cdots\subset\la^r=\la
\]
such that $|\la^i|=|\la^{i-1}|+1$ and $\la^i/\la^{i-1}$ is a 
$k$-horizontal strip for $1\leq i\leq r=|\la-\mu|$. The latter objects 
are exactly the standard $k$-tableaux on $\la/\mu$. This shows that 
(b) implies (c), and the converse is also clear.
\end{proof}

The previous results show that the non-zero terms in equations
(\ref{mastereq}) and (\ref{master3}) correspond exactly to the terms
in equations (\ref{besteq}) and (\ref{Fxx}), respectively, when 
$w=w_\la$ is a $k$-Grassmannian element of $B_\infty$.

\begin{cor}
\label{stdcor}
Let $w\in B_\infty$ be a skew element and $(\la,\mu)$ be a compatible
pair such that $w_\la = w w_\mu$.  Then the number of reduced words
for $w$ is equal to the number of standard $k$-tableaux on $\la/\mu$
and to the coefficient of $x_1x_2\cdots x_r$ in
$2^{-r}F^{(k)}_{\la/\mu}(x)$, where $r=|\la - \mu|=\ell(w)$.
\end{cor}

\begin{example}
Let $\la$ be a $k$-strict partition, $\la^1$ be defined as in \S
\ref{wsec}, $\la^2=\la\ssm\la^1$, and $\mu\subset\la^2$. We can form a
standard $k$-tableau on $\la/\mu$ by filling the boxes of $\la^2/\mu$,
going down the columns from left to right, and then filling the boxes
of $\la^1$, going across the rows from top to bottom. If $\mu\subset\la$
but $\mu_1 > k$, this procedure does not always work, see for instance
Example \ref{4ex}(c). When $k=3$, $\la = (8,6,5,2)$, and
$\mu=\emptyset$, the $3$-tableau on $\la$ which results is
\[
\begin{array}{cccccccc}
1 & 5  & 9  & 12 & 13 & 14 & 15 & 16 \\
2 & 6  & 10 & 17 & 18 & 19 & & \\ 
3 & 7  & 11 & 20 & 21 &  & & \\
4 & 8  &  & &  &  &  & 
\end{array}
\]
which corresponds to the reduced word
\[
s_1\,s_0\,s_2\,s_1\,s_0\,s_4\,s_3\,s_2\,s_1\,s_0\,
s_3\,s_2\,s_1\,s_5\,s_4\,s_3\,s_2\,s_6\,s_5\,s_4\,s_3
\]
for the Grassmannian element $w_\la=167\ov{5}\ov{3}\ov{2}4\in B_7$.
\end{example}

A sequence $(i_1,\ldots,i_m)$ is called {\em unimodal} if for some $r$
with $0\lequ r \lequ m$, we have $i_1 > i_2 > \cdots > i_r < i_{r+1} <
\cdots < i_m$. An element $w\in B_\infty$ is unimodal if it has a
reduced word $s_{i_1}\cdots s_{i_m}$ such that $(i_1,\ldots,i_m)$ is a
unimodal sequence. A tableau $T$ has {\em content} given by the
composition $\al$ if $\al_i$ of the entries of $T$ are equal to $i$,
for each $i\geq 1$. We can now state the following generalization of
Corollary \ref{stdcor}.

\begin{prop}
\label{uni}
Let $w\in B_\infty$ be a skew element and $(\la,\mu)$ be a compatible
pair such that $w_\la = w w_\mu$. Then there is a 1-1 correspondence
between reduced factorizations $u_1\cdots u_r$ of $w$ into unimodal
elements $u_i$ and $k$-tableaux $T$ of shape $\la/\mu$ with $r$
distinct entries. The lengths of the $u_i$ agree with the content of
$T$, and the number of reduced words for $w$ obtained by concatenating
unimodal reduced words for $u_1,\ldots, u_r$ of the corresponding
lengths is equal to $2^{n(T)-r}$.
\end{prop}
\begin{proof}
From the definition of $F_w$ in \S \ref{ssfns} it follows that if
there is only one variable $x$ and $w\neq 1$, then we have
$F_w(x)=2\,n_w\,x^{\ell(w)}$, where $n_w$ denotes the number of
unimodal reduced words for $w$. Moreover, for each $m\geq 1$ we have
\begin{equation}
\label{Fxm}
F_w(x_1,\ldots,x_m) = \sum_{u_1\cdots u_m=w}F_{u_1}(x_1)\cdots F_{u_m}(x_m)
\end{equation}
summed over all reduced factorizations $u_1\cdots u_m$ for $w$. The
result follows by comparing (\ref{Fxm}) with Definition \ref{Fdefn}
and using Example \ref{4ex}(d) and Corollary \ref{compcor}.
\end{proof}

\noin 
Let $\la\in \cP(k,n)$ and $w_\la\in B_n$ be the corresponding
$k$-Grassmannian element. There are analogues of Corollary
\ref{stdcor} and Proposition \ref{uni} for $\Ti_\la(x\,;y)$ and the
$k$-bitableaux of shape $\la$. We say that a permutation $v\in S_n$ is
{\em decreasing} if $v$ has a reduced word $s_{i_1}\cdots s_{i_m}$
such that $i_1> \cdots >i_m$. Then the $k$-bitableaux of shape $\la$
correspond to reduced factorizations $u_1\cdots u_rv_1\cdots v_s$ of
$w_\la$ with the $u_i\in B_n$ unimodal and the $v_j\in S_n$
decreasing. We leave the details to the reader.

According to \cite{BH} and \cite{L}, the type C Stanley symmetric
function $F_w$ is a nonnegative integer linear combination of Schur
$Q$-functions.  Together with Theorem
\ref{abprop}, this implies the following result.

\begin{cor}
\label{bhcor}
For any two $k$-strict partitions $\la$, $\mu$ with $\mu\subset\la$, 
the function $F^{(k)}_{\la/\mu}(x)$ is a nonnegative integer linear
combination of Schur $Q$-functions.
\end{cor}

\subsection{}
We say that a permutation $\om$ is {\em fully commutative} if any
reduced word for $\om$ can be obtained from any other by a sequence of
braid relations that only involve commuting generators.  It follows
from \cite[Thm.\ 4.2]{Ste} that a permutation $\om$ is fully
commutative if and only if there exists a Grassmannian permutation
$\om_\la$ and a reduced factorization $\om_\la = \om\om'$ for some
permutation $\om'$. In other words, the skew elements of the symmetric
group are exactly the fully commutative elements.

Following \cite{FS}, the expansion of the formal product
\[
A_1(x_1)A_1(x_2)\cdots = \sum_{\om\in S_n}G_\om(x)\, u_\om
\]
may be used to define the type A Stanley symmetric functions
$G_\om(x)$ for $\om\in S_n$. Stanley \cite{Sta} introduced $G_\om$ to
study the set of reduced words for $\om$ (he actually worked with
$G_{\om^{-1}}$). It is shown in \cite[\S
2]{BJS} that if $\om\in S_n$ is fully commutative (or equivalently,
{\em 321-avoiding}) then $G_\om = s_{\la/\mu}$ is a skew Schur
function, and the number of reduced words for $\om$ is equal to the
number of standard tableaux on $\la/\mu$. Our present definition and
study of skew elements in the hyperoctahedral group is therefore
completely analogous to this established theory for the symmetric
group.

Any fully commutative element of $B_n$, in the sense of \cite{Ste}, is
a skew element. However the converse is emphatically false, for
example the $1$-Grassmannian element $w_{(4,1)}= 2\ov{3}1$ is not
fully commutative. The three reduced words 
\[
s_1\,s_2\,s_1\,s_0\,s_1 \qquad
s_2\,s_1\,s_2\,s_0\,s_1 \qquad s_2\,s_1\,s_0\,s_2\,s_1
\]
for $2\ov{3}1$ correspond respectively to the three standard $1$-tableaux
\[
\begin{array}{l} 1234 \\ 5 \end{array} \qquad \ \
\begin{array}{l} 1245 \\ 3 \end{array} \qquad \ \
\begin{array}{l} 1345 \\ 2 \end{array}
\]
on the diagram $\la=(4,1)$.

\end{document}